\documentclass{amsart}
\usepackage[utf8]{inputenc}

\usepackage[margin=1.15in]{geometry}
\usepackage{amsmath,amsfonts,amsthm,amssymb}

\usepackage{color,comment,graphicx,pinlabel}
\usepackage{soul,xcolor,placeins}
\usepackage{tikz-cd}
\usepackage{thmtools}
\usepackage{thm-restate}
\usepackage{mathtools}
\usepackage{hyperref}
\usepackage{cleveref}
\usepackage{todonotes}
\usepackage{microtype}

\newtheorem{theorem}{Theorem}[section]
\newtheorem{corollary}[theorem]{Corollary}
\newtheorem{lemma}[theorem]{Lemma}
\newtheorem{proposition}[theorem]{Proposition}
 
\newtheorem{question}[theorem]{Question} 
 
\theoremstyle{definition}
\newtheorem{definition}[theorem]{Definition}
\newtheorem{remark}[theorem]{Remark}
\newtheorem{example}[theorem]{Example}

\DeclareMathOperator{\Ends}{Ends}

    \author[Mark Hughes]{Mark Hughes}
    \address{Brigham Young University\\Provo, UT, 84602 USA}
    \email{hughes@mathematics.byu.edu}
    
    \author[Alexandra Kjuchukova]{Alexandra Kjuchukova}
    \address{University of Notre Dame\\Notre Dame, IN, 46556 USA}
    \email{akjuchuk@nd.edu}
    
    \author[Maggie Miller]{Maggie Miller}
    \address{University of Texas at Austin\\Austin, TX, 78712 USA}
\email{maggie.miller.math@gmail.com}

\title{Branched Covers of Open Manifolds\\
\vspace{-.15in}}

\begin{document}

\vspace*{-.1in}

\thanks{MH was supported by an NSF grant DMS-2213295; AK by an NSF grant DMS 2204349;  MM by a Clay Research Fellowship, NSF grant DMS-2404810 and Simons Foundation Gift MPS-TSM-00007679.}

\begin{abstract} For $m=2$ and $m=3$ we prove that any connected, oriented, open manifold 
$M^m$ admits a simple branched covering map over  $\mathbb{R}^m$. When $M$ has $k$ ends and $k$ is finite, the degree of the cover can be taken to be $mk$. Regardless of the number of ends, $M$ admits a branched covering map of countably infinite degree over $\mathbb{R}^m$. 
We also investigate which compact manifolds 
are \textit{universal bases}, that is, are branch covered by all compact manifolds in the same dimension.
\end{abstract}

\maketitle

\vspace*{-.3in}
\section{Introduction}\label{sec:intro}

In this paper manifolds are, with a few explicit exceptions, orientable; but they may or may not be compact. We are interested in manifolds $M^m$ which have the following universal property: every $m$--manifold is a branched cover of $M$. We call such manifolds {\it universal bases}. It has been known since Alexander~\cite{alexander1920note} that spheres are universal bases in the PL category in all dimensions. A half century later it was shown that $S^3$ is universal even if we restrict to 3--fold covers with connected branching sets~\cite{hilden1974every, hirsch1974offene, montesinos1974representation}; and that assuming a manifold $M^3$ is a homotopy 3-sphere actually suffices to conclude that $M$ is a universal 3-base~\cite{piergallini1992covering}. For $S^4$, one needs degree 4 to achieve universality and degree 5 suffices for the branching sets to be embedded~\cite{iori2002}.  

We are interested in understanding what other universal bases exist. One family of examples of universal bases in dimension $m$ are orientable manifolds whose universal cover is $S^m$. Thus, in odd dimensions, lens spaces are examples of universal bases. Furthermore, with some work one can deduce from~\cite{montesinos2002representing} that $\mathbb{R}^3$ is a universal base among open manifolds; and the same is shown for  $\mathbb{R}^4$ in~\cite{ piergallini2019branched}. Both proofs rely on compactifying the spaces, and the local picture of the branched cover in neighborhoods of the ends is not explicit. 
To our knowledge, this is the first investigation of the question: which $m$--manifolds are universal bases? We restrict to manifolds which have empty boundary (but are not necessarily compact).

When approaching universality, it makes sense to constrain the degree of the covering maps. Simply put, higher degree maps increase the potential for universality: even spheres $S^m$ are not universal bases in degree less than $m$. There are also some basic assumptions without which universality is a vacuous concept. We propose the following definition.

\begin{definition} \label{def:uni--base} Fix an integer $n\ge 2$, and let CAT denote some compatible subset of the adjectives \emph{compact}, \emph{open}, \emph{closed}, and \emph{orientable}. A CAT $m$--manifold $M$ is a {\it universal $n$--base} if for every connected CAT $m$--manifold $N$, there exists an $n$--fold branched covering $f: N\to M$.  When $M$ is open, we additionally say that $M$ is a \emph{universal $\aleph_0$--base} if every open $m$--manifold $N$ admits a countably infinite-sheeted branched covering $f: N\to M$.
\end{definition}

Once the properties of $M$ are specified, the restrictions we impose on $N$ are natural: if $M$ is open, closed or orientable, then so is $N$. When $M$ is noncompact, then since $N$ surjects to $M$ we must also have $N$ noncompact.

On the other hand,  when $M$ is non-orientable it is in principle possible that all (open or closed, as $M$) manifolds, including orientable and non-orientable ones, admit degree--$n$ branched covering maps over $M$. However, in Corollary~\ref{cor:nonorientodd} we prove that no compact non-orientable $n$-bases exist.

We classify compact universal bases in dimensions 2 and 3 and prove that $\mathbb{R}^2$ and $\mathbb{R}^3$ are universal bases using arguments that do not rely on compactification.  Our results can be summarized as follows: 

\begin{theorem}\label{thm:summary}\leavevmode
    \begin{enumerate}
        \item For any $n\in\mathbb{N}$, the only closed 2--dimensional universal $n$--base is $S^2$.
        \item A closed orientable 3--manifold $M$ is a universal base if and only if $M$ is spherical.
        \item $\mathbb{R}^2$ and $\mathbb{R}^3$ are universal bases. 
    \end{enumerate}
\end{theorem}

The classification of open universal bases among 2-- and 3--manifolds remains open.

\subsection*{Organization.}

In the remainder of Section \ref{sec:intro}, we discuss branched covering maps in more detail and give expanded statements of our main theorems (from which Theorem \ref{thm:summary} follows). 
Our results on closed manifolds are collected in Section~\ref{sec:closed}, and those on open manifolds 
in Section~\ref{sec:noncomp}. In Section~\ref{sec:open-qs}, we highlight some open problems and we construct explicit 2-- and 3--fold branched covers of the Whitehead manifold over $\mathbb{R}^3$.

\subsection{Definition and local description of a branched cover}

Our preferred definition of branched covers is in the PL category. Since our main results are in dimensions 2 and 3, where PL and smooth coincide, we generally do not distinguish between the two. 
For a brief discussion of the ways that a branched covering map can be defined in different contexts, we refer the reader to the appendix at the end of this paper. We use the following definition:

\begin{definition}\label{def:branchedcovering}
Given two (possibly noncompact, non-orientable) connected $m$--dimensional PL manifolds $M$ and $N$ with empty boundary, a nondegenerate map $f: N\to M$ is a {\it branched cover} if there is a nonempty codimension-2 subpolyhedron $B$ of $M$ such that the restriction of $f$ to a map $N\setminus f^{-1}(B)$ $\rightarrow M \setminus B$ is a $d$--fold covering map. We refer to $d$ as the degree, $B$ as the {\it branch locus}, and to $f^{-1}(B)$ as the  {\it branch set} of $f$.
If $M$ is compact, we require that $N$ is compact and hence $d$ is finite. When $M$ is noncompact, we also let $d=\aleph_0$ for countably infinite-sheeted covers. 
If $N$ is disconnected, a map $f: N\to M$ is a branched cover if its restriction to each component of $N$ meets the above criteria.
\end{definition}

Given $f, M, N, B$ and $d$ as above, Piergallini~\cite{piergallini1989manifolds} showed that for any point $b\in B$ there is a sufficiently small open ball $U$ around $b$ which can be parametrized as $U\cong \mathbb{R}^{m-2}\times\mathbb{C}$ such that on each component $U_i$ of $f^{-1}(U)$, the restriction $f|_{U_i}$ is either given by $(x,z)\mapsto (x,z^r)$ for some $r\in\mathbb{N}$ or is the cone on a $r$--sheeted branched covering of $S^{m-1}$ over $S^{m-1}$ (here, $m=\dim M=\dim B+2$). The integer $r$ is called the \emph{branching index} or \emph{local degree} of $b$ at $f^{-1}(b)\cap U_i$, and can vary with $i$. When $N$ is connected, the sum of local degrees at $b$ equals the total degree $d$ for all $b\in B$.

In the statements and proofs below, we will often focus on simple branched covering maps. A $d$--fold branched covering map $f$ is said to be {\emph{simple}}  if 
every point $b$ in the branch set has one preimage of index 2 and ($d-2$) preimages of index 1. Equivalently, the meridian of $b$ is mapped to a transposition in the symmetric group $S_d$ under the map $\pi_1(M\backslash B)\to S_d$ which determines the cover. In dimensions 2 and 3 a branched covering map is generically simple~\cite[Proposition~3.3 and Theorem~6.5]{bersteinontheconstruction}, and simple covers form an open set in the space of branched covers between compact $m$-manifolds for all $m$~\cite[Proposition~3.1]{bersteinontheconstruction}. When possible in our constructions we will explicitly produce simple covering maps. 

\subsection{Main results}

It is well-known and easy to show that $S^2$ is a universal $2$--base. In fact, $S^2$ is the only universal base among closed surfaces: see 
Lemma~\ref{lemma:betti} and Proposition~\ref{cor:surfaces-nonuni}.  We study the case of noncompact orientable surfaces in Section~\ref{sec:surfaces}, where we prove that $\mathbb{R}^2$ is a universal $2k$--base for orientable surfaces with $k\in \mathbb{N}$ ends (Theorem~\ref{thm:surfacenoncompact}), and a universal $\aleph_0$--base among all open surfaces. 

As noted earlier, by a famous result of Hilden, Hirsch, and Montesinos, $S^3$ is a universal 3--base~\cite{hilden1974every, hirsch1974offene, montesinos1974representation}, while $S^4$ is a universal 4--base \cite{piergallini1995four}. As covers over spheres may be stabilized in degree, these results also imply that $S^3$ is a universal $n$--base for all $n\ge 3$, while $S^4$ is a universal $n$--base for all $n\ge 4$.

\begin{restatable}{theorem}{noncompact}
\label{thm:noncompact} 
Let $N$ be an orientable open 3--manifold with $k$ ends. Then $N$ admits a simple $n$--fold branched covering map over $\mathbb{R}^3$, with $n=\min\{3k,\aleph_0\}$. In other words, $\mathbb{R}^3$ is a universal $3k$--base among open manifolds with $k$ ends, and is a universal $\aleph_0$--base among open manifolds with infinitely many ends.
\end{restatable}

 One approach to proving Theorem \ref{thm:noncompact} would be the following. First, apply work of Montesinos \cite{montesinos2002representing}, in which he shows that every planar 3--manifold $N$, i.e.\ $S^3$ with a nonempty subset $E$ of the Cantor set deleted, is a universal 3--base among open 3--manifolds with end space homeomorphic to $E$. 
 This means that $N$ admits a degree--3 branched covering map over such a manifold. Then, to prove that $\mathbb{R}^3$ is universal, it would suffice to show that all planar 3--manifolds are branched covers of it.

 Our approach differs fundamentally from that in \cite{montesinos2002representing}, which relies on abstract principles to extend branched covering maps of open manifolds over their compactifications. The arguments in the current paper are constructive, relying only on the existence of a compact exhaustion of an open manifold. When we need exhaustions with certain properties, we will describe how to achieve them. One advantage is that this strategy can be tractable for producing examples --- for instance, we give the explicit branch locus for 2-- and 3--fold branched coverings of the Whitehead manifold over $\mathbb{R}^3$ in Example~\ref{ex:whitehead}.

\begin{restatable}{corollary}{rthreeuniversal}
\label{rthreeuniversal}
  The space $\mathbb{R}^3$ is a universal $\aleph_0$--base.
\end{restatable}

Note Corollary \ref{rthreeuniversal} realizes open 3--manifolds with a potentially uncountable number of ends as $\aleph_0$ covers of $\mathbb{R}^3$. 
We also prove the analogue of Theorem \ref{thm:noncompact} in dimension two. 

\begin{restatable}{theorem}{surfacenoncompact}\label{thm:surfacenoncompact}
Let $\Sigma^2$ be an open surface with $k$ ends. Then, there exists  a simple $n$--fold branched covering map $\Sigma^2\to\mathbb{R}^2$, with $n=\min\{2k,\aleph_0\}$.
\end{restatable}
Piergallini--Zuddas \cite[Theorem 1.8]{piergallinizuddas_open} show that the analogous statement also holds in dimension~4.

\begin{restatable}{corollary}{rtwouniversal} \label{r2--universal}
    The plane $\mathbb{R}^2$ is a universal $\aleph_0$--base.
\end{restatable}

\section{Closed universal $n$--bases} \label{sec:closed}

We establish some restrictions on the algebraic topology of closed universal $n$--bases. 
We then consider non-orientable universal bases, contrasting even and odd dimensions. We also prove Proposition~\ref{prop-compact}, which states that  a closed 3--manifold is a universal $n$--base (for some $n$) if and only if its fundamental group is finite.

\begin{lemma}\label{lemma:betti}
If either $S^m$ or $\mathbb{R}^m$ is a finite--fold branched cover of a manifold $M^m$, then 
$b_i(M)=0$ for all $0<i<m$. In particular, if  $M^m$ is a universal $n$--base for some $n$, then we must have $b_i(M)=0$ for all $0<i<m$. 
\end{lemma}

\begin{proof}
    Suppose there is a finite--fold covering $f:X\to M$ of degree $d$ where $X$ is either $S^m$ or $\mathbb{R}^m$, and that
    $b_j(M)>0$ for some $0<j<m$. Let $\alpha\in H_j(M;\mathbb{Z})$ be a class of infinite order. After possibly replacing $\alpha$ by a multiple of $\alpha$, the class $\alpha$ can be represented by a $j$--dimensional submanifold $W$ of $M$, by \cite[\S2 Theorem II.29]{novikov2007topological} (a translation of~\cite{thom1954quelques,  thom1, thom2}; see also~\cite{sullivan}). Isotope $W$ to be transverse to the branch locus of $f$, and let $\widehat{W}$ be the preimage of $W$ under $f$. Since $d$ is finite, $\widehat{W}$ is a compact submanifold of $X$ and $f|_{\widehat{W}}:\widehat{W}\to W$ is a $d$--fold branched cover. 
    Since $0<j<m$, we have $H_j(X;\mathbb{Z})=0$. So, there is a $(j+1)$--chain $\rho:\sqcup\Delta^{j+1}\to X$ whose boundary is $\widehat{W}$. Then $f\circ \rho:\sqcup\Delta^{j+1}\to M$ has boundary $dW$, contradicting the fact that $c \alpha\neq 0$ for all $c\neq 0$. Hence, $b_i(M)=0$ for $0<i<m$. 
\end{proof}

\begin{corollary}\label{cor:nonorientodd}
There are no closed non-orientable universal $n$-bases for any $n \in \mathbb{N}$.
\end{corollary}

\begin{proof}
Let $M$ be a closed non-orientable $m$--dimensional manifold.  We'll prove the corollary separately for $m$ even and $m$ odd. 

The argument for $m$ even is a generalization of one found in \cite{bais2025branched}.  Let $f:N \rightarrow M$ be a $d$--fold branched covering.  Let $w_1(M) \in H^1(M; \mathbb{Z}_2)$ and $w_1(N)\in H^1(N; \mathbb{Z}_2)$ be the first Stiefel-Whitney classes of $M$ and $N$ respectively, and $[M]\in H_m (M; \mathbb{Z}_2)$ and $[N]\in H_m (N; \mathbb{Z}_2)$ the respective fundamental classes.  
    
Note that $f^*(w_1(M)) = w_1(N)$.  The most geometric way to see this is by noting that the first Stiefel-Whitney class $w_1(X) \in H^1(X; \mathbb{Z}_2) \cong Hom_{\mathbb{Z}_2} (H_1(X; \mathbb{Z}_2),\mathbb{Z}_2)$ is characterized by the fact that if $\gamma : S^1 \to X$ is a loop in $X$, then $w_1(X)([\gamma]) = 0$ if $\gamma$ is an orientable loop, and 1 otherwise.  For any loop $\gamma: S^1 \to N$ we may isotope $\gamma$ so that its image is disjoint from $B$ and so that $f \circ \gamma$ is an embedding.  Then it is clear that $\gamma$ will be orientable if and only if $f \circ \gamma$ is, since their images are identified homeomorphically by $f$.  

Then
\begin{align*}
\langle w_1(N)^m,[N]\rangle &=\langle f^*(w_1(M))^m, [N]\rangle \\ &= \langle w_1(M)^m, f_*([N])\rangle \\
& = \langle w_1(M)^m,d[M]\rangle \\ & = d \langle w_1(M)^m,[M]\rangle 
\end{align*}
where all of the evaluations are considered in $\mathbb{Z}_2$.  If $\langle w_1(M)^m,[M]\rangle =0$, then any manifold $N$ with $\langle w_1(N)^m,[N]\rangle = 1$ therefore cannot be a $d$--fold cover (for example, $N = \mathbb{RP}^m$). If $\langle w_1(M)^m,[M]\rangle =1$, then $\langle w_1(N)^m,[N]\rangle$ must have the same parity as $d$.  Since $\langle w_1(N)^m,[N]\rangle$ may be 0 or 1 in even dimensions this proves that there are no non-orientable universal $n$--bases when $m$ is even.
    
The above proof can break down in some odd dimensions. For example, all 3--manifolds $N$ bound 4--manifolds, and hence we have $\langle w_1(N)^m,[N]\rangle = 0$ for all such $N$.  Instead, we use the following argument: when $m$ is odd, we have  $\chi(M)=0$, and, since both $b_0(M) =1$ and $b_m(M)=0$, we must have $b_j(M)>0$ for some $0<j<m$. Then, by Lemma~\ref{lemma:betti}, $S^m$ is not a finite--fold branched cover of $M$, so $M$ is not a universal $n$--base.
\end{proof}

\begin{proposition}\label{cor:surfaces-nonuni}
    If $\Sigma$ is a closed surface that is a universal $n$--base for some $n$, then $\Sigma\cong S^2$.
\end{proposition}

Note that since any closed genus--$g$ surface $F$ admits a degree--2 branched cover over $S^2$ (with $2g+2$ branch points), by stabilizing the cover one can find a degree--$d$ branched covering map from $F$ to $S^2$ for any $d\geq 2$. Therefore, $S^2$ is in fact a universal $n$--base for {\emph{every}} $n \geq 2$ among closed, orientable manifolds. 

\begin{proof}[Proof of Proposition \ref{cor:surfaces-nonuni}]
    The condition $b_1(\Sigma)=0$ rules out every other surface except $\mathbb{RP}^2$. We will show that there is no integer $n$ such that every closed, non-orientable surface admits a degree--$n$ branched cover over $\mathbb{RP}^2$. Hence, the projective plane is not a universal $n$--base for any $n$.
    
    As an aside, note that every closed surface does admit a branched cover \emph{of some degree} over $\mathbb{RP}^2$. For orientable surfaces, this holds since they admit branched covers over $S^2$, which itself covers $\mathbb{RP}^2$. So, in particular, every orientable surface admits a degree--4 branched covering map over $\mathbb{RP}^2$. To see that the non-orientable surface $\#_h\mathbb{RP}^2$ covers $\mathbb{RP}^2$, we can write $\#_h\mathbb{RP}^2$ as $S^2\#_h\mathbb{RP}^2$, where the connected sum disks are made equivariant with respect to the $h$--fold cyclic branched cover $S^2\to S^2$. The $\mathbb{Z}/h\mathbb{Z}$ action on this surface permutes the $\mathbb{RP}^2$ summands and has quotient $\mathbb{RP}^2$. This constructs a degree--$h$ covering map $\#_h\mathbb{RP}^2 \to \mathbb{RP}^2$ with two branch points. Alternatively, given the base $\mathbb{RP}^2$ with two branch points $p,q$, note that $\mathbb{RP}^2\setminus\nu(\{p,q\})$ is homeomorphic to the boundary sum of a M\"{o}bius band and an annulus. This cover is associated with the homomorphism $\pi_1(\mathbb{RP}^2\setminus\{p,q\})\to\mathbb{Z}/h\mathbb{Z}$ given by mapping a core of the M\"{o}bius band to 0 and a core of the annulus to 1. The resulting cyclic branched cover is connected, non-orientable, and has Euler characteristic $h(1-2)+2=2-h$, and hence is homeomorphic to $\#_h\mathbb{RP}^2$.

    Returning now to the task of showing that we cannot fix a single degree for all such branched covers, let $F_h$ denote the non-orientable genus--$h$ surface (i.e. $F\cong \#_h\mathbb{RP}^2$) and let $f: F_h\to \mathbb{RP}^2$ be a degree--$d$ branched cover. We give a direct proof of the claim, also implied by Corollary~\ref{cor:nonorientodd}, that $h\equiv d\mod 2$. By possibly perturbing $f$ while preserving degree, we may assume that $f$ is a simple branched cover ~\cite[Proposition~3.3]{bersteinontheconstruction}.  
    The advantage of passing to a simple cover is that it makes it easy to the compute the Euler characteristic of the total space $F_h$. Recall that a point $x$ in the branch locus of a degree--$d$ simple cover has one preimage of index 2 and $(d-2)$ preimages of index $1$. 
    Let $b$ denote the number of branch points of $f$. Thus, if there are $b$ points in the branch locus of $f$ in $\mathbb{RP}^2$, we have $$\chi(F_h)=d(1-b) + (d-1)b=d-b = 2-h,$$ so $d+h=2-b$. Since the cover is simple of degree $d$, by definition the meridian of each branch point is sent to a transposition in the symmetric group $S_d$. Moreover, the square of the generator of $\pi_1(\mathbb{RP}^2)$ is sent to the product of the meridians of the points in the branch locus. In order for this product to be a square, $b$ must be even. Therefore, the equation $d+h=2-b$ implies $h\equiv d\mod 2$, as claimed. It follows that $\mathbb{RP}^2$ is not a universal $n$--base for any $n$. 
\end{proof}

Proposition~\ref{cor:surfaces-nonuni} shows that the conclusion of Corollary \ref{cor:nonorientodd} also holds in the case $m=2$. As far as we are aware, it is not known whether an $m$--dimensional non-orientable manifold can be a universal $n$--base (for some $n$) for even $m>2$. 

If a closed $m$--dimensional manifold $M^m$ is a universal $n$--base, we can draw a stronger conclusion than $b_1(M)=0$.

\begin{lemma}\label{lem:finite-pi-1}
If $M$ is a closed $m$--dimensional manifold for which there exists a branched covering map $S^m \rightarrow M$, then $\pi_1(M)$ is finite.
\end{lemma}

\begin{proof}
    Assume that there is a branched covering map  $f: S^m\to M$ of degree $d$. Then, $f$ lifts to a map from $S^m$ to the universal cover $\widetilde{M}$. But then, if $\widetilde{M}$ is noncompact, the induced map $H_m(S^m;\mathbb{Z})\to H_m(M;\mathbb{Z})$ factors through $H_m(\widetilde{M};\mathbb{Z})=0$, contradicting that the map $H_m(S^m;\mathbb{Z})\rightarrow H_m(M;\mathbb{Z})$ is multiplication by $d$. Hence, $\pi_1(M)$ is finite.
\end{proof}

   The following proposition is a converse to the above statement in the case of $m=3$.

\begin{proposition}\label{prop-compact}
    A closed 3--manifold $M^3$ is a universal $n$--base for at least one $n\in\mathbb{N}$ if and only if $M$ has finite fundamental group. In this case, $M$ is a universal $n$--base for $n=3|\pi_1(M)|$.
\end{proposition}

\begin{proof} 
    Let $M^3$ be a closed universal $n$--base. By Lemma~\ref{lem:finite-pi-1}, the fundamental group $\pi_1(M)$ is finite.

    For the converse, assume $|\pi_1(M)|<\infty$. Then $M$ is spherical and the universal covering map $S^3\to M$ has degree $|\pi_1(M)|$.  By~\cite{hilden1974every, hirsch1974offene, montesinos1974representation}, every closed manifold $N^3$ is a 3--fold branched cover of $S^3.$ The composition of these branched covers with the universal covering map are $3|\pi_1(M)|$--fold branched covers of $M$, and hence $M$ is a universal $3|\pi_1(M)|$--base.
  \end{proof}

\begin{remark}
A closed 3--manifold $M$ has finite fundamental group if and only if $M$ is spherical~\cite{perelman2002entropy,perelman2003ricci, perelman2003finite}, meaning that $M$ admits a complete metric of curvature +1, or equivalently when the universal cover of $M$ is $S^3$. A census of finite 3--manifold groups can be found in~\cite[Chapter 1]{aschenbrenner20153}.
\end{remark}

Note that, when $m>3$, it is not known whether $\pi_1(M^m)$ being finite implies that $M$ is a universal $n$--base. For $m=4$, the argument of Proposition~\ref{prop-compact} does not apply, as the universal cover of $M^4$ need not be $S^4$. For $m>4$, it is even open, to our knowledge, whether the $m$--sphere is a universal $n$--base for some fixed $n$, even as it is classically known that every PL $m$--manifold is a branched cover of the $m$--sphere of some unspecified degree~\cite{alexander1920note}.

It is also an open question whether $n=3|\pi_1(M)|$ in Proposition~\ref{prop-compact} is minimal. In what follows, we obtain some restrictions on $n$ by examining the order of elements in the fundamental group of a universal $n$--base. The following is proved by a similar argument to the one used in establishing Lemma~\ref{lemma:betti}.

\begin{lemma}\label{lem:order} Let $M^m$ be a universal $n$--base. For any $g\in\pi_1(M)$, the order of $g$ in $\pi_1(M)$ is at most $n$. 
\end{lemma}

\begin{proof}  If $m=2,$ by our earlier discussion $M$ is $S^2$ or $\mathbb{R}^2$, so the conclusion is automatic. Now assume $m\geq 3.$ Consider the $n$--fold branched cover $f: X^m\to M^m$, where $X = S^m$ if $M$ is compact, and $X = \mathbb{R}^m$ otherwise. Denote the branch locus of $f$ by $B^{m-2}$ and pick a basepoint $x_0$ in $M\backslash B$. Any element $g\in \pi_1(M, x_0)$ can be represented by an embedded circle $\gamma$ disjoint from $B$. Therefore, $f|_{f^{-1}(\gamma)}: f^{-1}(\gamma)\to \gamma$ is an $n$--fold cover of $S^1$, so $f^{-1}(\gamma)$ is a disjoint union of $r$ circles $S^1_1, \dots, S^1_r$. That is, for $i=1, 2, \dots, r,$ the  restriction   $f_i: S^1_i\to \gamma$ is a cover of degree $d_i$, where $\Sigma_i d_i=n$. Let $d_j = \min\{d_1, \dots, d_r\}$. We easily see that the order of $g$ in $\pi_1(M)$ is at most $d_j$: let $\rho: D^2\to X^m$ be a map with $\rho( \partial D^2)=S^1_j$, then the image of $f\circ\rho: D^2\to M^m$ has boundary $\gamma^{d_j}$. Since $d_j\leq n$ and $g\in \pi_1(M)$ was arbitrary, the result follows.  
\end{proof}

We also show that, given a manifold $M$, the order of the fundamental group of $M$ restricts the possible values of $n$ for which $M$ may be a universal $n$--base.

\begin{lemma}\label{lemma:easy-lower-bound}
    Let $M^m$ be a closed manifold with $|\pi_1(M)|=\ell$ finite. Any branched cover $S^m\to M^m$ has degree at least $\ell$. In particular, $M$ is not a universal $n$--base for any $n<\ell$.
\end{lemma}

\begin{proof}
    As before, it suffices to consider $m\geq 3$. Assume there exists a branched cover $f:S^m\to M^m$ of degree $d<\ell$, with branch locus $B^{m-2}\subset M^m$. Pick a basepoint $x\in (M\backslash B)$. Let $g_1, g_2, \dots, g_\ell$ be embedded loops representing the $\ell$ distinct elements of $\pi_1(M; x)$. Let $\tilde{x}$ denote one of the $d$ elements of the set $f^{-1}(x)$. Additionally, denote by $\tilde{g}_i$ the lift of $g_i$ starting at $\tilde{x}$, and by $\tilde{x}_i$ the endpoint of this lift. By assumption, $d<\ell,$ which implies that for some pair $i\neq j$ we have $\tilde{x}_i=\tilde{x}_j$.  But then $\tilde{g}_i\cdot \tilde{g}_j^{-1}$ is a loop in $S^m$ and in particular it is nullhomotopic. Composing such a nullhomotopy with $f$ contradicts $[g_i]\neq [g_j]$.
\end{proof}

When the fundamental group of a manifold $M^m$ is cyclic, we obtain a further condition on the integers $n$ for which $M$ can be $n$--fold covered by~$S^m$.

\begin{lemma}\label{lemma:divisibility}
Let $M^m$ be a closed oriented manifold with cyclic fundamental group of order $\ell$. The degree of a branched cover $S^m\to M$ with nonempty branch locus is strictly greater than $\ell$ and a multiple of $\ell$.
\end{lemma}

\begin{proof}
    Again, we have $m\geq 3$. Suppose $f: S^m\to M^m$ is a branched cover of degree $d$ with nonempty branch locus $B^{m-2}\subseteq M^m$. Pick a basepoint $x\in B$ and let $\alpha$ be an embedded circle representing a generator of $\pi_1(M; x)\cong \mathbb{Z}/\ell\mathbb{Z}$ which has no intersections with $B$ other than the basepoint. The lift of $\alpha$ to $S^m$ is a graph with $|f^{-1}(x)|=:d_0$ vertices and $d$ edges.  Because $x\in B$, we have that $d_0<d$. Moreover, each vertex in this graph has even degree (equal to twice the local branching degree at each lift of $x$), so there is an Euler circuit, $\tilde{\alpha}$, which is composed of $d$ edges. But since the graph has fewer vertices than edges, the Euler circuit revisits at least one vertex at least once, before traversing every edge. Thus, there is a loop $\tilde{\alpha}_0$ comprised of $c$ edges of this graph and such that $c\leq d_0<d$. Both $\tilde{\alpha}$ and $\tilde{\alpha}_0$ are nullhomotopic, being loops in $S^m$. Therefore, $f(\tilde{\alpha})=d\alpha$ and   $f(\tilde{\alpha}_0)=c\alpha$ are nullhomotopic as well. Since $\alpha$ has order $\ell$, it follows that $\ell$ divides both $c$ and $d$. Therefore, $\ell\leq c\leq d_0 <d=r\ell$ for some $r\in\mathbb{N}$. Since $d>\ell,$ we have $r\geq 2$. 
\end{proof}

One is tempted to surmise that the inequality $r\geq m$ probably  holds, but we have not established this statement even in the case $m=3$.

\section{Open universal bases} \label{sec:noncomp}

This section contains our results about branched coverings between noncompact manifolds.  
 
 \subsection{Dimension-independent preliminaries}
We recall the definition of an \emph{end} of a topological space. For a thorough discussion see~\cite{hughes1996ends}.

\begin{definition}\label{def:ends}
Let $M$ be a noncompact space. An {\it end} of $M$ is an equivalence class of sequences of connected open
neighborhoods $M\supset U_1 \supset U_2\supset \dots$ with the property that $\bigcap_{i=1}^\infty \overline{U}_i=\emptyset,$ where $$(M\supset U_1 \supset U_2\supset \dots)\sim (M\supset V_1 \supset V_2\supset \dots)$$ if for each $U_i$ there exists a $j$ such that $V_j\subseteq U_i$, and for each $V_j$ there exists an $\ell$ with $U_\ell \subseteq V_j$.
\end{definition}

The following definition is important in the study of open 3--manifolds. For instance, it is the canonical way in which nontrivial contractible 3--manifolds were proved to exist. We state this definition here to emphasize that distinct ends of an open manifold may display extremely different behavior --- there is no reason to expect a symmetry between the ``different infinities'' of an open manifold. Note that the following definition is specific to a choice of end, in contrast to ``the fundamental group at infinity,'' which only makes sense for one-ended open manifolds.

\begin{definition}\label{def:pi-1-end}
    The {\emph{fundamental group of an end}} $M\supset U_1 \supset U_2\supset \dots$ is the inverse limit of groups $\displaystyle\varprojlim_{i}\pi_1(U_i)$.  
\end{definition}

We will refer to cardinality of the set of the equivalence classes in Definition~\ref{def:ends} as the \emph{number of ends} of $M$. Note that when $M$ is a manifold we can regard the open sets $U_i$ above as the complements of the elements of a \emph{compact exhaustion} of $M$, that is, a nested sequence $K_0\subseteq K_1\subseteq K_2\dots$ of compact sets such that  
$K_i\subseteq K_{i+1}^\circ$ and $\bigcup_{i=1}^\infty K_i=M$.

We consider branched coverings of the form $f: M^m\to N^m$, where $M$ and $N$ are noncompact $m$--manifolds with empty boundary.
As we will see shortly, by relating the number of ends of $M$ and $N$ to the degree of $f$, we can conclude that there are no universal noncompact $n$--bases with $n\in\mathbb{N}$.  However, by restricting the number of ends of $M$ we can still obtain some universality results.  

\begin{proposition}\label{prop:ends-ineq}
    Let $M$ and $N$ be noncompact manifolds with $k$ and $\ell$ ends, respectively, where both $k$ and $\ell$ are finite. If there exists a $d$--fold branched cover $N\to M$, then $\ell\leq dk$.
\end{proposition}

\begin{proof}
We begin with a compact exhaustion of $M$ by compact sets $K_0\subseteq K_1\subseteq K_2\subseteq\cdots \subseteq M$, where the $K_j$ are chosen so that each connected component of $M\setminus K_j$ is not contained in a compact subset of $M$. (If a component $U$ of $M \setminus K_j$ has compact closure, then for all $i\ge j$ replace $K_i$ with $K_i \cup \overline{U}$. Note that after a finite number of steps all $K_i$ will remain unchanged by this procedure.) If $M$ has $k$ ends, this implies that $M\setminus K_j$ has exactly $k$ connected components for sufficiently large~$j$.

Now suppose $f:N\to  M$ is a $d$--fold branched cover, and let $E_i=f^{-1}(K_i)$. Then $E_0\subseteq E_1\subseteq \cdots$ is a compact exhaustion of $N$ as above. Since $f_|: N\setminus E_i \to M\setminus K_i$ is a $d$--fold branched cover, the preimage under $f$ of every connected component of $M\setminus K_i$ has at most $d$ components. For $i$ large enough, the components of $M\setminus K_i$ and $N\setminus E_i$ correspond to the ends of $M$ and $N$, respectively. Therefore, the number of ends of $N$ is at most $dk$. 
\end{proof}

\begin{corollary}
    Let $N$ be an $m$--manifold with $k$ ends. Then any branched cover $N^m \rightarrow \mathbb{R}^m$ has degree at least $k$.
\end{corollary}

\begin{corollary}
There does not exist an open universal $n$--base for any $n<\infty$. 
\end{corollary}

On the other hand, if we make the natural restriction on the number of ends, we do find some open universal $n$--bases with $n$ finite. Dimensions~2 and~3 are discussed in the rest of this section.  We moreover show in Corollaries~\ref{rthreeuniversal} and~\ref{r2--universal} that $\mathbb{R}^2$ and $\mathbb{R}^3$ are universal $\aleph_0$--bases. (From \cite{piergallinizuddas_open}, it follows that $\mathbb{R}^4$ is also a universal $\aleph_0$--base.)

\subsection{Noncompact surfaces} \label{sec:surfaces}

In this section, we will constructively prove Theorem \ref{thm:surfacenoncompact}.

\surfacenoncompact*

We first discuss an interesting consequence, Corollary~\ref{r2--universal}, which also relies on the following preliminary proposition.

\begin{proposition}\label{prop:r2coveringr2}
    For any $m\ge 2$, There exists a simple $\aleph_0$--fold branched covering map from $\mathbb{R}^m$ to $\mathbb{R}^m$, with branch locus a countably infinite collection of $\mathbb{R}^{m-2}$s.
\end{proposition}

\rtwouniversal*

\begin{proof}[Proof of Corollary~\ref{r2--universal}  
from Theorem \ref{thm:surfacenoncompact} and Proposition \ref{prop:r2coveringr2}]
    Let $\Sigma$ be an open surface. If $\Sigma$ has infinitely many ends, then by Theorem \ref{thm:surfacenoncompact}, $\Sigma$ admits an $\aleph_0$--fold branched cover over $\mathbb{R}^2$. If $\Sigma$ has finitely many ends, then by the same theorem there exists a finite--fold branched cover $f:\Sigma\to\mathbb{R}^2$. Let $g:\mathbb{R}^2\to\mathbb{R}^2$ be the $\aleph_0$--fold branched cover constructed in Proposition~\ref{prop:r2coveringr2}, with branch set disjoint from the branch locus of $f$. Then $g\circ f:\Sigma\to\mathbb{R}^2$ is a (potentially nonsimple) $\aleph_0$--fold branched cover. If desired, perturb $g\circ f$ to obtain a simple branched cover.
\end{proof}

\begin{proof}[Proof of Proposition \ref{prop:r2coveringr2}]
We will construct the required $\aleph_0$--fold cover $\mathbb{R}^2\to \mathbb{R}^2$, which will have simple branch points at the integer values along the positive $x$-axis.  Begin by fixing a basepoint $(1,0)\in \mathbb{R}^2$. For all $j \in \mathbb{N}$, let $\gamma_j$ denote the meridian about $(0,j)$ represented by the obvious loop whose image in the plane is the circle of radius $\frac{1}{4}$ centered at $(0,j)$, together with the straight line from (1,0) to the point $(\frac{1}{4}, j)$. Then $\{[\gamma_j]\mid j\in\mathbb{N}\}$ generates $\pi_1(\mathbb{R}^2\setminus\{(0,i)\mid i\in\mathbb{N}\})$. Let $S_{\mathbb{N}}$ denote the symmetric group on $\mathbb{N}$, i.e.\ the group that permutes the labels in the set $\{1,2,3,\ldots\}$. Now let $\rho$ be the homomorphism from $\pi_1(\mathbb{R}^2\setminus\{(0,i)\mid i\in\mathbb{N}\})$ to $S_{\mathbb{N}}$ determined by mapping $[\gamma_j]$ to the transposition exchanging $j$ and $j+1$. This determines a simple $\aleph_0$--fold branched cover $f$ over $\mathbb{R}^2$, with total space $\mathbb{R}^2$ as required.

Finally, for any $m\ge 2$, let $F:\mathbb{R}^m\to\mathbb{R}^m$ be obtained by taking the product of $f$ with the identity map on $\mathbb{R}^{m-2}$. The resulting $F$ is an $\aleph_0$--fold simple branched cover.
  \end{proof}
  
  One can, perhaps, better visualize the covering constructed in Proposition~\ref{prop:r2coveringr2} by thinking of it in terms of compact exaustions of the base and total space.  To do this, decompose the base and total space $\mathbb{R}^2$s  as nested unions of closed disks $D^2_i$ and $E^2_i$ respectively, for $i\in \mathbb{N}$.  More precisely, let let $D_i$ and $E_i$ be the closed disks centered at the origin, of radius $i+\tfrac{1}{2}$, where we think of $D_i$ as living in the base space $\mathbb{R}^2$ while $E_i$ lives in the total space $\mathbb{R}^2$.  For each $i$, there is an $(i+1)$--fold covering map $f_i: E_i\to D_i$ which is determined by labeling the point $(0,j)$  in the base space with the transposition $(j,j+1) \in S_{i+1}$, for $1 \leq j \leq i$.  These covering maps respect the natural inclusions $D_i \subset D_{i+1}$ and $E_i \subset E_{i+1}$, in the sense that $f_i: E_i \to D_i$ is the restriction to $E_i$ of  $f_{i+1}: E_{i+1} \to D_{i+1}$.  It follows then that the $f_i$ define an $\aleph_0$--fold cover $\mathbb{R}^2\to \mathbb{R}^2$.  We illustrate this construction in Figures~\ref{fig:rcoveringr} and \ref{fig:rcoveringrside}. Above each point in the branch locus lies countably infinitely many points in the branch set, one in each $E_i$, illustrated as one row of points in Figure \ref{fig:rcoveringr} (left). In this row, $f$ has local degree two at exactly one point in the branch set, and degree one at all other branch points.

\begin{figure} [h!]
\centering
\labellist
\pinlabel{\scriptsize{2}} at 85 109
\pinlabel{\scriptsize{3}} at 85 80
\pinlabel{\scriptsize{4}} at 85 52
\pinlabel{\scriptsize{5}} at 85 21.5
\pinlabel{\scriptsize{6}} at 85 -9
\pinlabel{$f$} at 510 380
\endlabellist
\includegraphics[width=100mm]{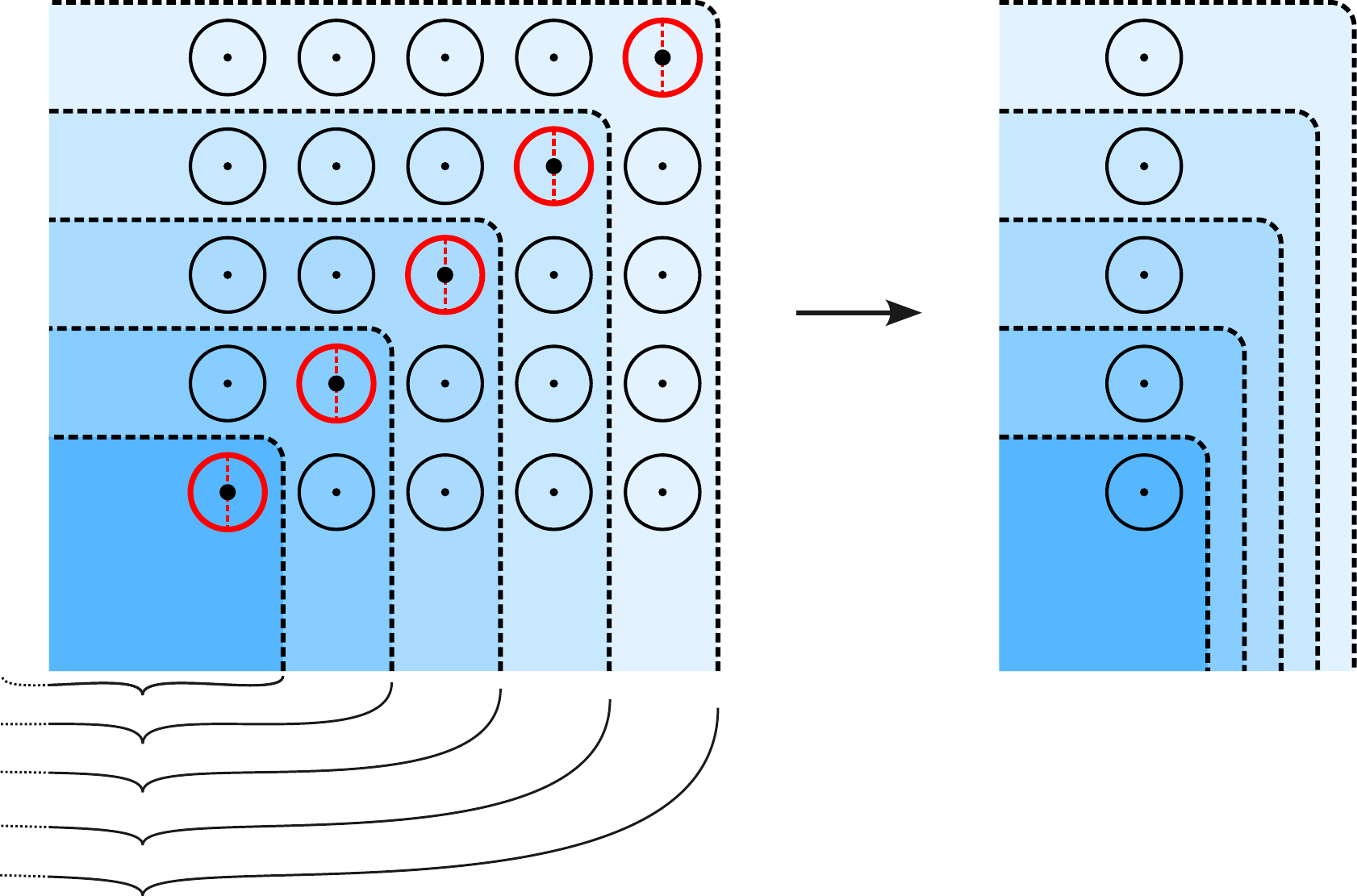}
    \caption{The $\aleph_0$--fold branched covering constructed in Proposition~\ref{prop:r2coveringr2}. This is a branched covering map $f$ from $\mathbb{R}^2$ to $\mathbb{R}^2$. The branch set (left) and branch locus (right) each contain countably infinitely many isolated points. One row of such points on the left cover a single point on the right. Above each point in the branch locus lies one point at which $f$ has degree two, indicated in bold. At all other points in the branch set, $f$ has degree one. Thus, $f$ is a simple branched cover. The overlapping highlighted portions are the elements $E_i, D_i$ of the two exhaustions. We indicate the degree of the restriction $f_i: E_i\to D_i$. The total map is an $\aleph_0$--fold cover. Note that while each point in the image of $f$ has infinitely many preimages, only finitely many can be contained in a bounded subset, as $f$ is a proper map.}\label{fig:rcoveringr}
\end{figure}

\begin{figure}
    \centering
    \labellist
    \pinlabel{$E_3$} at 285 215
    \pinlabel{$E_2$} at 170 178
    \pinlabel{$E_1$} at 90 138
    \pinlabel{$D_3$} at 285 25
    \pinlabel{$D_2$} at 170 25
    \pinlabel{$D_1$} at 90 25
     \pinlabel{$f_1$} at 48 62
       \pinlabel{$f_2$} at 138 62
       \pinlabel{$f_3$} at 250 62
    \endlabellist

\includegraphics[width=0.5\textwidth]{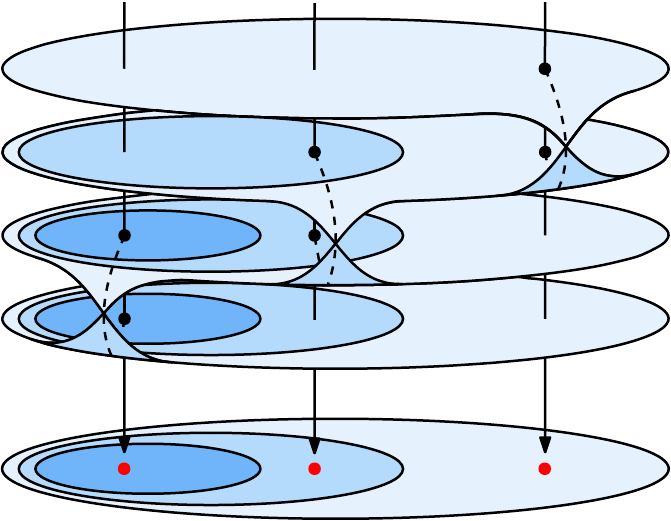}

\caption{A side view of the $\aleph_0$--fold branched covering constructed in Proposition~\ref{prop:r2coveringr2}.  The base $\mathbb{R}^2$ is composed of the union of the disks $D_i$ at the bottom of the figure, while the total space $\mathbb{R}^2$ is represented by the stack of disks above it. The branch locus is represented by the red dots in the base space, and the various $D_i$ and $E_i$ disks are differentiated by their relative shades of blue. The preimage of a point in the branch locus is a countably infinite set of points lying directly above it. Vertical pairs of black dots in the total space are identified. When traveling along a loop around a branch point, crossing a dashed lines corresponds to switching sheets in the total space.}
\label{fig:rcoveringrside}
\end{figure}

We now turn to the proof of Theorem \ref{thm:surfacenoncompact}. Our strategy is to decompose an open surface $\Sigma$ into basic pieces that admit compatible simple branched covers over $\mathbb{R}^2$. We make the following statements dimension-independent, so that we may use them in a 3--dimensional setting later.

\begin{proposition}\label{prop:surfacestandardex}
    Let $\Sigma$ be an open $m$--manifold. 
    Then $\Sigma$ admits a compact exhaustion $E_1\subset E_2\subset \cdots$ such that $E_1$ is an $m$-ball, and for all $j>1$ every connected component of $\overline{E_j\setminus E_{j-1}}$ is one of the following:
    
   \noindent {\it (a)} An $m$--manifold with two boundary components: one in $\partial E_j$ and one $\partial E_{j-1}$. 
    
   \noindent  {\it (b)} An $m$--manifold with three boundary components: two in $\partial E_j$ and one in~$\partial E_{j-1}$.
\end{proposition}

\begin{remark} \label{rem:counting-ends}
In Proposition \ref{prop:surfacestandardex}, since a component of $\overline{E_j\setminus E_{j-1}}$ has exactly one boundary in $\partial E_{j-1}$, an end of $\Sigma$ exactly corresponds to a sequence $C_1,C_2,\ldots$, where $C_j$ is a boundary component of $E_j$ that lies in the same component of $\overline{E_{j}\setminus E_{j-1}}$ as $C_{j-1}$. When the component of $\overline{E_{j}\setminus E_{j-1}}$ containing 
$C_{j-1}$ has just two boundary components, then $C_j$ is determined by $C_{j-1}$. In contrast, when the component of $\overline{E_{j}\setminus E_{j-1}}$ containing $C_{j-1}$ has three boundary components, then there are two choices for $C_j$. Thus, when there are a finite number $k$ of components $E_j$ with three boundaries, the number of ends of $\Sigma$ is precisely $k+1$. When there are infinitely many components of $\overline{E_{j}\setminus E_{j-1}}$ (across all $j$) with three boundary components, then $\Sigma$ has infinitely many ends.
\end{remark}

\begin{proof}[Proof of Proposition \ref{prop:surfacestandardex}]
    Begin with an arbitrary compact exhaustion $E_1\subseteq E_2\subseteq\cdots \subseteq \Sigma$. Let $D$ be a closed ball in the interior of $E_1$. Shift the indices of the $E_j$ each up by one and redefine
    $E_1:=D$.

Now for each $j=2,3,\ldots$, we perform the following sequence of operations. 
Note 
that after we finish performing these operations for some value of $j$, the compact set $E_j$ remains unchanged at all subsequent steps.
\begin{figure}
    \centering
    \labellist
    \pinlabel{$\gamma$} at 178 373
    \pinlabel{$E_j$} at 180 85
    \pinlabel{$E_{j+1}$} at 70 230
        \pinlabel{$C_2$} at 205 151
    \pinlabel{$C_1$} at 55 151
    \pinlabel{$E_{j+2}$} at 85 340
    \pinlabel{$E_j$} at 600 85
    \pinlabel{$E_{j+1}$} at 490 230
    \pinlabel{$E_{j+2}$} at 505 340
    \endlabellist
\includegraphics[width=100mm]{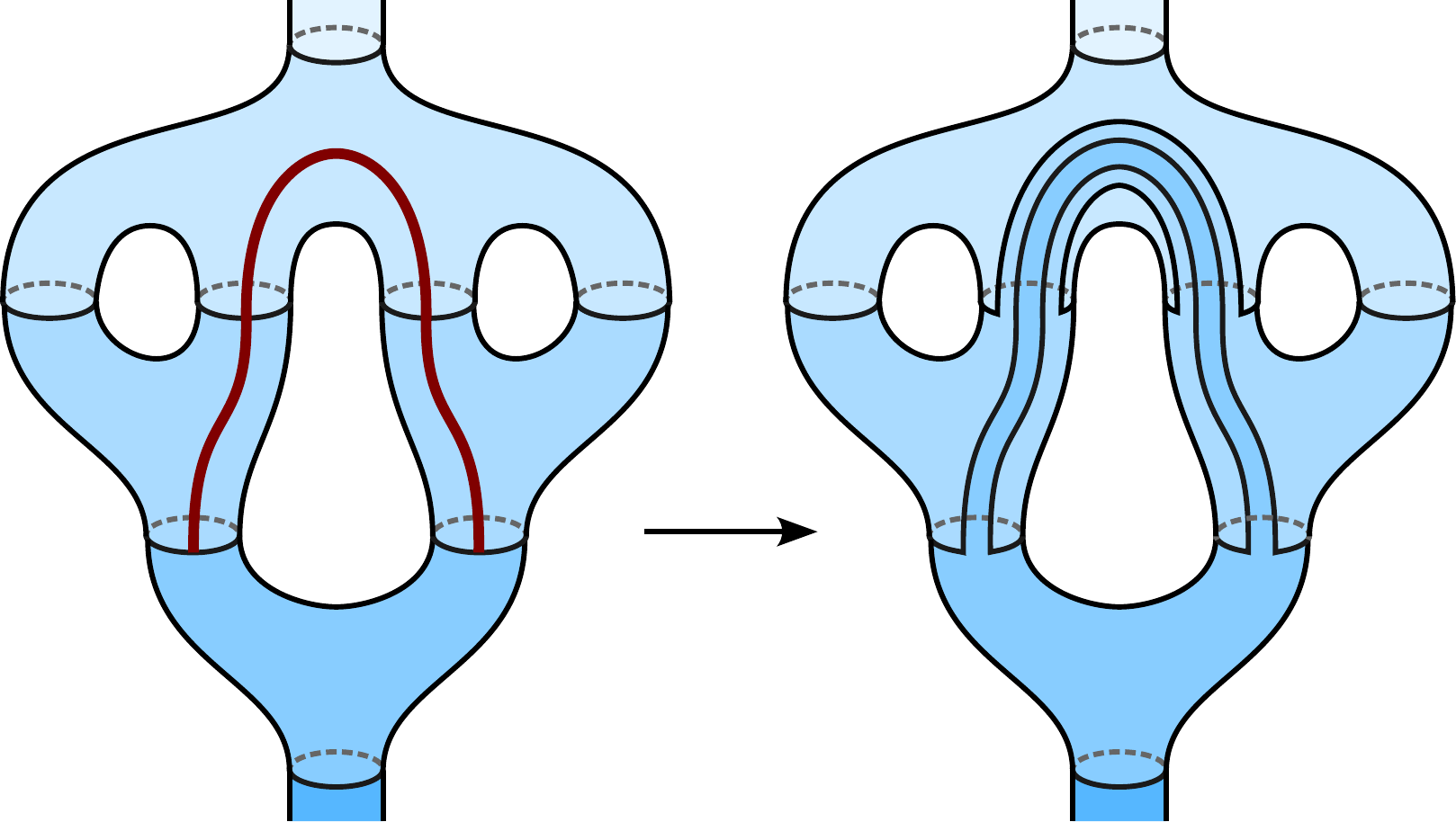}\caption{Left: a portion of an exhaustion $E_1\subset E_2\subset\cdots$ of an open manifold $\Sigma$. Here, $E_j$ has two boundary components $C_1, C_2$ that are in the same component of $\Sigma\setminus\mathring{E_j}$. We indicate an arc $\gamma$ in $\Sigma\setminus\mathring{E_j}$ connecting $C_1$ and $C_2$. In this example, $\gamma$ is contained in $E_{j+2}$. On the right, we form a new exhaustion in which replace $E_j$ and $E_{j+1}$ with $E_j\cup A_0$ and $E_{j+1}\cup A_1$ respectively, where $A_0,A_1$ are closed tubular neighborhoods of $\gamma$ with $A_0\subset\mathring{A}_1$ and $A_1\subset\mathring{E}_{j+2}$. This operation decreases the number of boundary components of $E_j$ by one.}
\label{fig:oneboundaryperend}
\end{figure}
\begin{enumerate}
    \item (Schematically illustrated in Figure \ref{fig:oneboundaryperend}.) Suppose $C_1,C_2$ are distinct boundary components of $E_j$ that lie in the same component of $\Sigma\setminus\mathring{E_j}$. Let $\gamma$ be an arc whose interior is contained in $\Sigma\setminus\mathring{E_j}$, from a point in $C_1$ to a point in $C_2$. Let $r$ be the smallest natural number so that $\gamma$ is contained in the interior of $E_{j+r}$. 
    Let $A_0\subsetneq A_1\subsetneq \cdots \subsetneq A_{r-1}$ be closed tubular neighborhoods of $\gamma$, with $A_{r-1}$ in the interior of $E_{n+r}$. For $i=0,1,\ldots, r-1$, redefine $E_{j+i}:=E_{j+i}\cup A_i$. This preserves the property that $E_i\subset E^\circ_{i+1}$ for all $i$.  Moreover, the operation just performed decreases the number of boundary components of $E_j$ by one. Repeat until no two boundary components of $E_j$ lie in the same component of $\Sigma\setminus\mathring{E_j}$.  In particular, this ensures that the number of boundary components of $\Sigma\setminus\mathring{E_j}$ can not decrease as $j$ increases.
    \item (Schematically llustrated in Figure \ref{fig:threeboundary}.) Suppose $C_1,C_2,C_3$ are boundary components of $E_j$ that lie in the same component of $\overline{E_j\setminus E_{j-1}}$. Let $P$ be a submanifold of the connected manifold $\overline{E_j\setminus E_{j-1}}$ obtained by taking a regular neighborhood of $C_2\sqcup C_3 \cup\gamma$, where $\gamma$ is an arc in $E_j$ connecting $C_2$ and $C_3$. (In dimension two, $P$ is a pair of pants.) Let $A:=\overline{E_j\setminus P}$. Shift the indices of $E_{j},E_{j+1},\ldots$ each up by one and redefine $E_j:=A$. This has the effect of decreasing by one the number of boundary components of $E_j$, and after repeated application we obtain a compact set $E_j$ which has three boundary components, distributed in a way which satisfies the conclusions of the proposition. 
\end{enumerate}
Repeating these two procedures sequentially on $E_j$, for $j = 2, 3, \ldots$ produces an exhaustion as described in the statement above.
\end{proof}

\begin{figure}
    \centering
    \labellist
    \pinlabel{$E_j$} at 130 90
    \pinlabel{$E_{j-1}$} at 130 -15
    \pinlabel{$E_{j+1}$} at 130 205
    \pinlabel{$E_j$} at 533 90
    \pinlabel{$E_{j-1}$} at 533 -15
    \pinlabel{$E_{j+1}$} at 533 140
    \pinlabel{$E_{j+2}$} at 533 205
    \endlabellist
    \vspace{.05in}
\includegraphics[width=100mm]{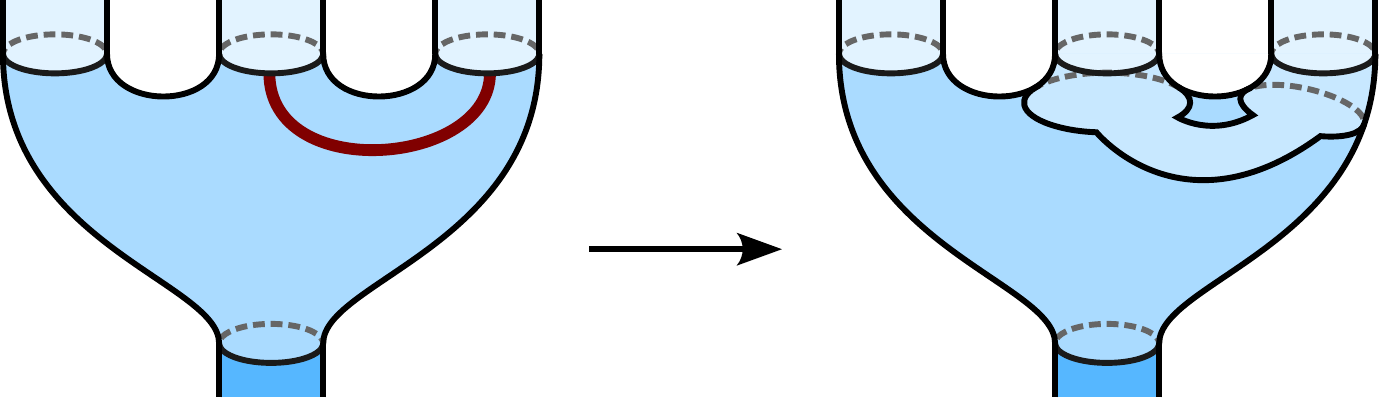}
\vspace{.05in}
\caption{Left: a portion of an exhaustion $E_1\subset E_2\subset\cdots$ of an open manifold $\Sigma$. Note that $E_j$ has three boundary components in the same component of $\overline{E_j\setminus E_{j-1}}$. We indicate an arc in $E_j\setminus E_{j-1}$ connecting two of these boundary components. Let $P$ denote a neighborhood of the arc and the two boundary components it meets, and let $A:=\overline{E_j\setminus P}$. On the right, we form a new exhaustion in which we shift the indices of $E_j, E_{j+1},\ldots$ up by one and set $E_j:=A$.}
\label{fig:threeboundary}
\end{figure}

\begin{figure}[h!]
    \centering
    \labellist
    \pinlabel{\Large$f$} at 375 230
    \endlabellist
    \includegraphics[width=100mm]{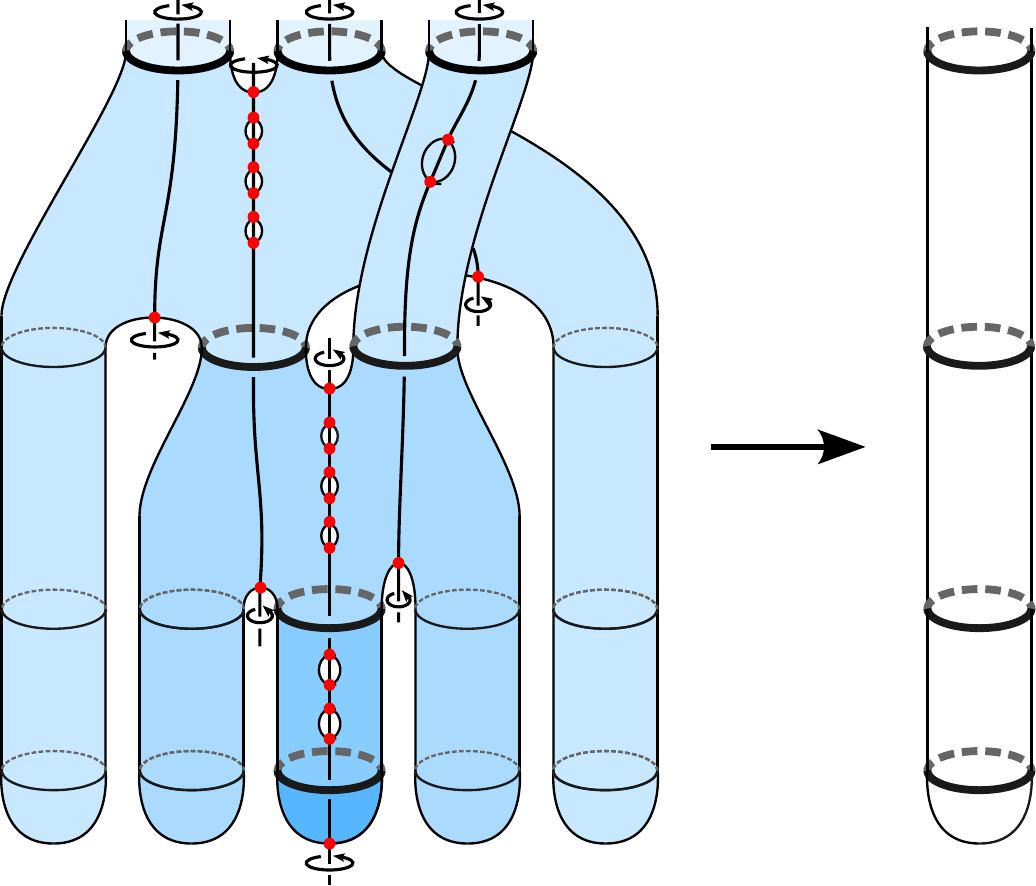}\caption{A simple branched cover $f$ of an open surface $\Sigma$ (left) over $\mathbb{R}^2$ (right). In this figure, a point $p\in\Sigma$ and $f(p)\in\mathbb{R}^2$ are represented by horizontal translates for all $p\in \Sigma$. In this schematic, $E_1$ is a dark disk (left surface, bottom row, center minimum), as in every exhaustion in our proof; $E_2$ is a genus--2 surface with one boundary component  (left surface, bottom two rows, center); $E_3$ has genus 5 and two boundary components; and so on. The elements of the exhaustion are indicated by the different shades.  The bold circles in $\Sigma$ are the boundary components of elements $E_1\subset E_2\subset\cdots$ of the exhaustion. Each of these circles double covers a circle in $\mathbb{R}^2$ (translating horizontally in the diagram). The nonbolded circles drawn in $\Sigma$ map homeomorphically onto the corresponding circles in $\mathbb{R}^2$. We indicate the degree--2 branch points in $\Sigma$ using red bolded points.}\label{fig:standardsurfaceform}
\end{figure}

\begin{proof}[Proof of Theorem \ref{thm:surfacenoncompact}]
    Let $E_1\subseteq E_2\subseteq \cdots$ be an exhaustion of $\Sigma$ satisfying the conditions of Proposition \ref{prop:surfacestandardex}. Let $D_1\subseteq D_2\subseteq\cdots$ be an exhaustion of $\mathbb{R}^2$ by nested disks. We construct a simple branched covering $f:\Sigma\to \mathbb{R}^2$ of degree equal to $\min\{2k,\aleph_0\}$, which we describe below.  The reader is encouraged to consult Figure~\ref{fig:standardsurfaceform}.

We define the map $f$ by describing its restriction on the components of $E_j \setminus E_{j-1}$.   On $E_1\cong D^2$, the restriction of the map $f:\Sigma \to \mathbb{R}^2$ is a 2--fold branched covering   $E_1 \to D_1$,  with one branch point.
    
    Now let $j \geq 2$, and let $A$ be a component of $\overline{E_j\setminus E_{j-1}}$ with two boundary components.  Then $f$ restricted to $A$ is a 2--fold branched cover over the annulus $\overline{D_j\setminus D_{j-1}}$, with $2g(A)$ branch points.

When $P$ is a component of $\overline{E_j\setminus E_{j-1}}$ with three boundary components then $f(P)=D_j$ (but {\emph{not}} as a branched covering map, as $f$ maps one of the components of $\partial P$ to the interior of $D_j$). Away from the branch points of $f$, each point in $D_j\setminus D_{j-1}$ has four preimages under $f$ in $P$, while points in $D_{j-1}$ have two preimages under $f$ in $P$. The map $f$ has $2g(P)+3$ index--2 branch points in $P$, each of which is mapped to a distinct point in $D_j\setminus D_{j-1}$.  Since all index--2 branch points of $f$ in $E_j\setminus E_{j-1}$ cover distinct points in $D_j\setminus D_{j-1}$, $f$ is a simple branched cover.

    Let $\mathcal{P}$ be the set of all connected components of $\overline{E_j\setminus E_{j-1}}$ that have three boundary components across all $j \geq 2$. Then for a regular point $p\in D_1$, the preimage $f^{-1}(p)$ contains two points in $E_1$ and two points in each $P\in\mathcal{P}$. By Remark~\ref{rem:counting-ends}, if $|\Ends(\Sigma)|=k$ is finite, then $|\mathcal{P}|=k-1$. On the other hand, if $|\Ends(\Sigma)|$ is infinite then $|\mathcal{P}|=\aleph_0$. Thus, the degree of $f$ is precisely $\min\{2k,\aleph_0\}$, as claimed.
\end{proof}

\begin{remark}
    The fact that a 2--manifold $\Sigma$ with uncountably many ends is a countable cover of $\mathbb{R}^2$ branched over a discrete set of points may raise some eyebrows which we hereby set out to lower.
    
    In the proof of Theorem \ref{thm:surfacenoncompact}, we arranged for ends of the surface $\Sigma$ to each correspond to a sequence of choices of boundary in components of $\overline{E_i\setminus E_{i-1}}$ of the form we called $P\in\mathcal{P}$.    Each such $P\subset \overline{E_i\setminus E_{i-1}}$ has three boundary components, two of which are in $\partial E_i$. Using arbitrary choice, label these two boundaries ``$L$'' and ``$R$'' for each $P\in\mathcal{P}$. Now an end of $\Sigma$ uniquely determines a (potentially infinite) word in the letters $\{L,R\}$, according to which curves intersect a ray that goes out to that end (without doubling back in the exhaustion). The set $\mathcal{P}$ also determines the degree, $2+2|\mathcal{P}|$, of $f$. Note the contrast: an end corresponds to a (finite or infinite) word in $\{L, R\}$, while the degree of the cover is a linear function of the order of a finite or countably infinite set and is thus finite or countable.

    If the end space is infinite, there is at least one end whose associated word in $\{L,R\}$ is also infinite; and if the end space of $\Sigma$ is uncountable, there are uncountably many such ends. 
    From this perspective one might consider the fact that it is possible for a countable-sheeted cover over $\mathbb{R}^2$ to have uncountably many ends to be as surprising as the fact that the set of infinite sequences in $\{L,R\}$ is uncountable. 
\end{remark}

\subsection{Noncompact 3--manifolds}
In this section, we prove our main result, reproduced below.  

\noncompact*

The theorem will be proved as follows. First we will choose a convenient exhaustion $E_1\subseteq E_2 \subseteq \dots$ of $N$ (i.e.\ apply Proposition \ref{prop:surfacestandardex}) and take $E_1$ to cover $B^3$ via a simple 3--fold branched covering map. We will then inductively extend this cover over $E_i$ with $i$ increasing, as in the proof of Theore~\ref{thm:surfacenoncompact}. In order to extend a given cover, we will rely on the following extension lemma. 

\begin{lemma}\label{lem:extend-from-bdry} 
Let $N^3$ be a connected, compact, oriented 3--manifold with boundary $A\sqcup B$, where $A$ and $B$ are nonempty surfaces which may be disconnected. Fix a positive integer $d\ge 3$ and simple $d$--fold covers $f_a:A\to S^2 \cong S^2 \times \{-1\}$, and $f_b:B\to S^2 \cong S^2 \times \{1\}$. Then $f_a,f_b$ extend to a simple branched cover $F:N\to S^2\times [-1,1]$.
\end{lemma}

Note that in Lemma \ref{lem:extend-from-bdry}, the simple branched covers $f_a, f_b$ may restrict to unbranched covers on some components of $A$ or $B$.  A version of Lemma~3.10 is proven for $N$ with two connected boundary components in \cite[Theorem 6.2]{berstein1979construction}.  For completeness, we include a proof that allows for arbitrarily many boundary components in the manifold $N$.

\begin{remark}
The reader may compare \cite[Lemma 3]{montesinos2002representing}, which is used in proving that compactified open 3--manifolds are branched covers of $S^3$. In \cite[Lemma 3]{montesinos2002representing}, it is established that a compact oriented 3--manifold $M^3$ with $\ell$ boundary components admits a branched cover over $S^3\setminus\sqcup_\ell \mathring{B}^3$ whose restriction to each boundary component of $M$ is a 3--fold branched cover over one of the 2--sphere boundaries of the target. 
The lemma does not guarantee that a {\emph{fixed}} cover $\partial M\to S^2$ can be extended to all of $M$, which is what we need (in addition to setting $\ell=2$ regardless of what $M$ is; and loosening the requirements on the degree of the given cover at each component of $\partial M$) in order to construct covers of open manifolds over $\mathbb{R}^3$ by piecing together maps on compact subsets.  
\end{remark}

\begin{proof}[Proof of Lemma \ref{lem:extend-from-bdry}]
We view $N$ as a cobordism from $A$ to $B$. As such, $N$ can be decomposed as a union of two compression bodies
\[H_a:= (A\times [-1,0])\cup \text{1-handles along }(A\times \{ 0\})\]
and 
\[
H_b:=(B\times [0,1])\cup \text{1-handles along }(B\times \{0\}).
\]
In other words, we have a decomposition of the cobordism $N$ as a union of 1-handles (those in $H_a$) and 2-handles (which, turned upside down, are 1-handles in $H_b$).

We can now extend $f_a,f_b$ separately over each compression body $H_a, H_b$ by adding one arc to the branch set within each attached 1-handle (see Lemma 6.1 in~\cite{bersteinontheconstruction} and also Figure~\ref{fig:1HandleAttach}; this is where we use that $N$ is orientable, as the 1-handle in question must be orientable for the lemma to apply). We briefly recall why extending $f_a$ over a 1-handle $h \cong D^2 \times [0,1]$ is possible. (The argument also applies to $f_b$.)  
Without loss of generality, we may assume that that the feet of the 1-handle are mapped to the same 2-disk, $D$, in $S^2$, which is disjoint from the branch locus; and furthermore that
\[
f_a((x,y) \times \{0\}) = f((x,-y)\times \{1\}),
\]
where we identified the feet of $h$, namely $D^2\times \{0\}$ and $D^2\times \{1\}$, with the disks in $A\times \{0\}$ to which they are attached.

Using the involution $(x,y) \times \{t\} \mapsto (x,-y) \times \{1-t\}$, we can extend $f_a$ locally over $h \cong D^2 \times [0,1]$. In order to complete the map thus obtained to a simple cover of the desired degree, we attach a 3-ball, $D^2\times[0, 1]$, along $D^2\times\{0\}$ to each additional component of pre-image $f_a^{-1}(D)$. We then extend the map over each such ball by a homeomorphism onto $f_a(h)$.
See Figures~\ref{fig:1HandlePre} and \ref{fig:1HandleAttach} for an illustration in degree 3.  We shall denote the extensions of $f_a,f_b$ over $H_a, H_b$ also by $f_a, f_b$, respectively.

\begin{figure}
\begin{center}
\includegraphics[width=85mm]{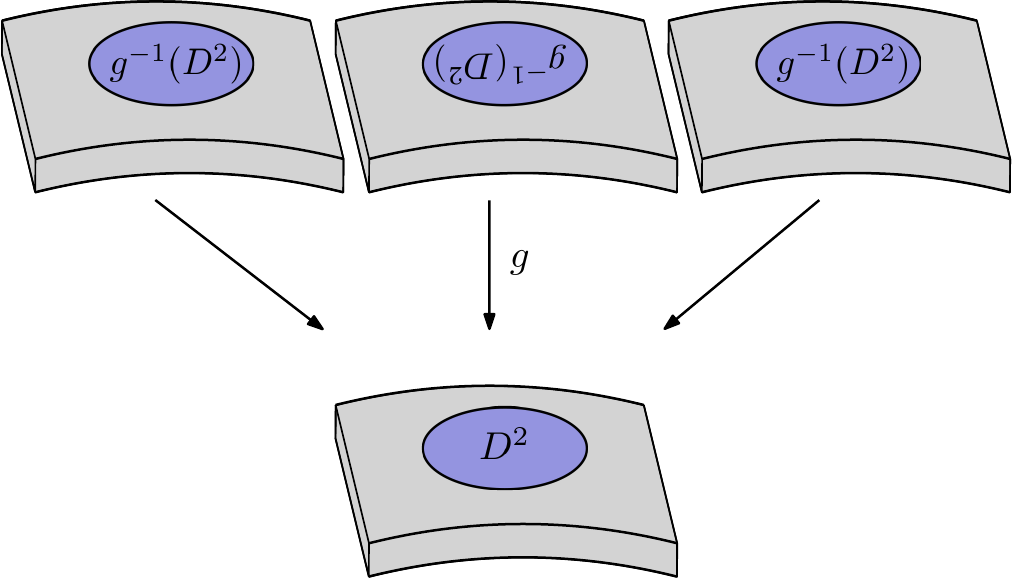}
\end{center}
\caption{Local model of simple branched covering away from the branching set. Here we draw a degree three cover.}
\label{fig:1HandlePre}

\begin{center}
\includegraphics[width=85mm]{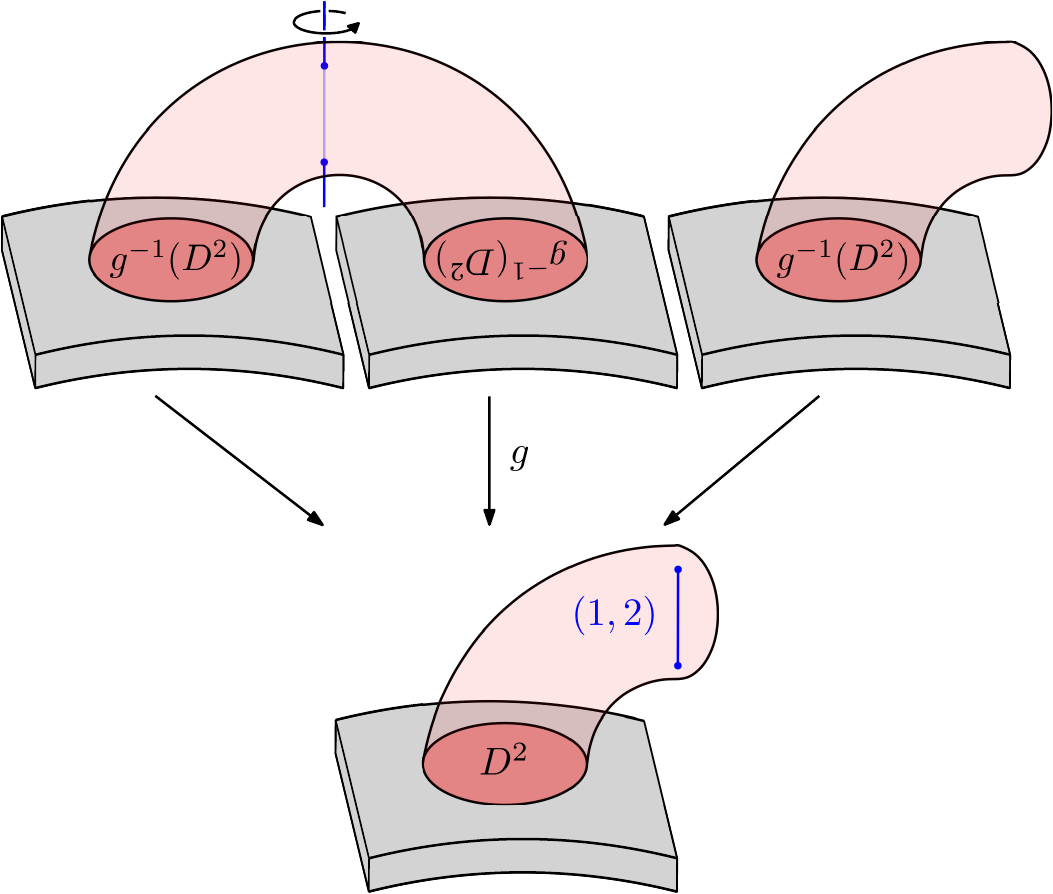}
\end{center}
\caption{Symmetrizing a 1-handle $H$ with respect to a simple branched covering $g$, so that $g$ can be extended along $H$ using the local involution $(x,y)\times\{t\} \mapsto (x,-y)\times\{1-t\}$.}
\label{fig:1HandleAttach}
\end{figure}

Now let $\Sigma = H_a \cap H_b$ be the orientable surface along which $H_a$ and $H_b$ are glued.  The maps $f_a$ and $f_b$ restrict to simple $d$--fold branched covering maps $\widetilde{f}_a,\widetilde{f}_b: \Sigma \rightarrow S^2$, which are not expected to agree. However, since the covers $\widetilde{f}_a$ and $\widetilde{f}_b$ have the same domain and degree, 
they are equivalent by \cite[Theorem 3.4]{bersteinontheconstruction}. That is, by composing with automorphisms of $\Sigma$ and $S^2$, we may assume that the two covers agree (see also~\cite{luroth} and~\cite{clebsch1873theorie}) and we will thus denote them both by $\widetilde{f}$.  Then, by \cite[Theorem 4.1]{bersteinontheconstruction} (here we use $d\ge 3$), there are homeomorphisms $\varphi : \Sigma \rightarrow \Sigma$ and $\psi: S^2 \rightarrow S^2$, both of which are isotopic to the identity function on their respective domains, which satisfy $\widetilde{f} \circ \varphi = \psi \circ \widetilde{f}$.  

As a result, if we instead glue $H_a$ and $H_b$ along $\Sigma$ by the map $\varphi$, and $S^2 \times [-1,0]$ to $S^2 \times [0,1]$ along $S^2 \times \{0\}$ by $\psi$, the branched coverings $\widetilde{f}_a$ and $\widetilde{f}_b$ glue to give a well-defined branched covering 
\[
F : H_a \cup_\varphi H_b \rightarrow  (S^2 \times [-1,0]) \cup_{\psi} (S^2 \times [0,1]).
\]
As both $\varphi$ and $\psi$ are isotopic to the respective identity maps, we have that $H_a \cup_\varphi H_b \cong N$ and $$(S^2 \times [-1,0]) \cup_{\psi}  (S^2 \times [0,1]) \cong S^2\times[-1,1],$$ so we have obtained the desired branched covering.
\end{proof}

\begin{proof}[Proof of Theorem~\ref{thm:noncompact}] 
Let $N$ be an open 3--manifold. 
    Let $E_1\subseteq E_2\subseteq \cdots$ be an exhaustion of $N$ satisfying the conditions of Proposition \ref{prop:surfacestandardex}. (This usage is the reason that Proposition \ref{prop:surfacestandardex} is written for general dimension.) Let $B_1\subseteq B_2\subseteq\cdots$ be an exhaustion of $\mathbb{R}^3$ by nested balls. We form a simple branched covering $f:N\to \mathbb{R}^3$ as in Figure \ref{fig:3mfdstandard}, with $E_i$ covering $B_i$ for each $i$, as follows.

    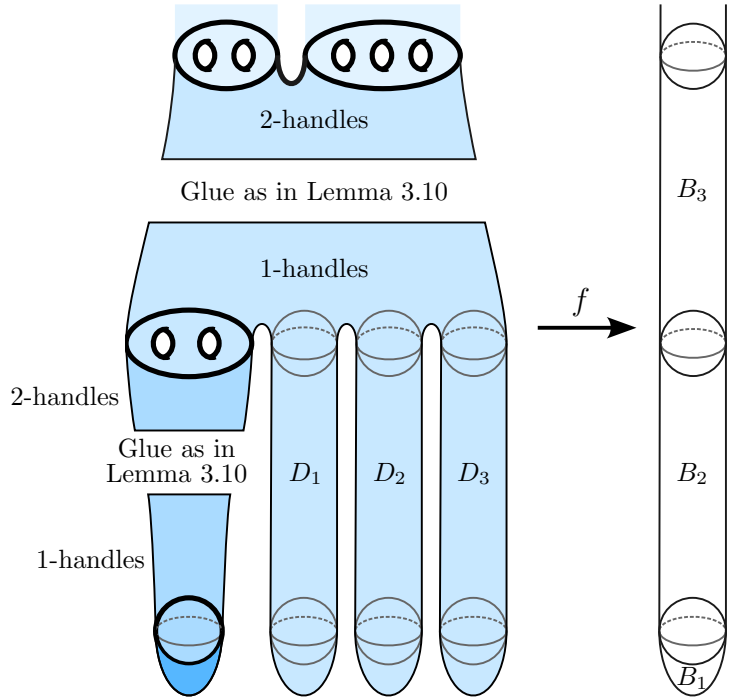
\begin{figure}
    \labellist
    \pinlabel{1-handles} at -20 80
    \pinlabel{2-handles} at -35 175
    \pinlabel{1-handles} at 110 250
    \pinlabel{2-handles} at 110 337
    \pinlabel{Glue as in Lemma \ref{lem:extend-from-bdry}} at 110 295
    \pinlabel{Glue as in} at 30 145
    \pinlabel{Lemma \ref{lem:extend-from-bdry}} at 30 130
    \pinlabel{$D_1$} at 105 130
    \pinlabel{$D_2$} at 154 130
    \pinlabel{$D_3$} at 203 130
    \pinlabel{$B_1$} at 330 11
    \pinlabel{$B_2$} at 330 130
    \pinlabel{$B_3$} at 330 295
    \pinlabel{\Large$f$} at 265 230
    \endlabellist
    \begin{center}
        \includegraphics[width=80mm]{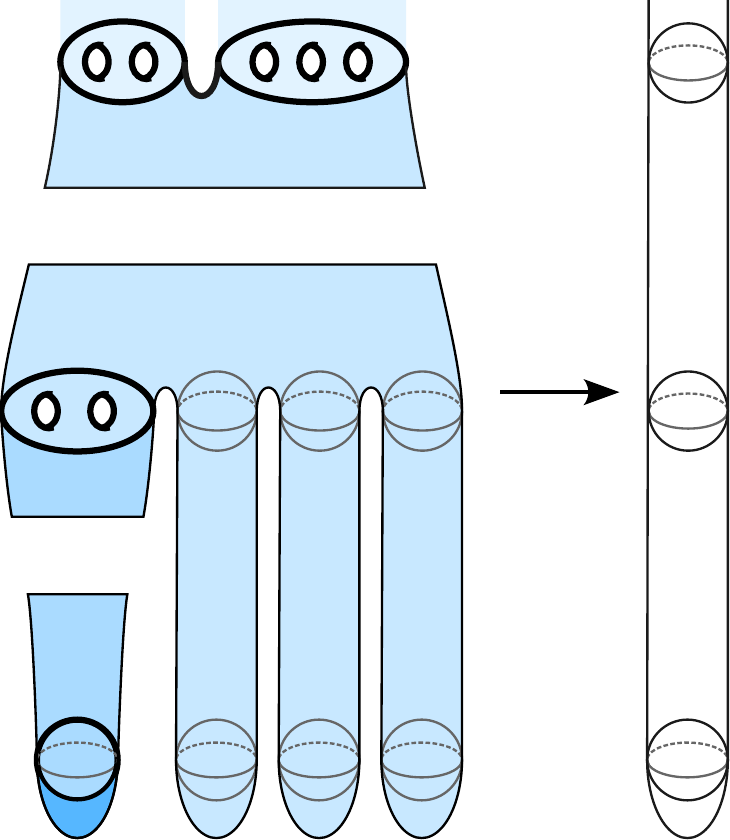}
        \caption{A simple branched cover $f$ of an open 3--manifold $N$ (left) over $\mathbb{R}^3$ (right). In this schematic, $E_1$ is a dark 3-ball; $E_2$ has one (genus--2) boundary component; $E_3$ has two boundary components (one genus--2 and one genus--3). The elements of the exhaustion are indicated by the different shades. In this figure, the indicated surfaces in $N$ on the left map to the 2--spheres on the right at the same height. The bold surfaces in $N$ are the boundary components of elements $E_1\subset E_2\subset\cdots$ of an exhaustion as in Proposition \ref{prop:surfacestandardex}. On each of these surfaces, $f$ restricts to a 3--fold simple branched cover over a 2--sphere in $\mathbb{R}^3$ (indicated by translating horizontally in the diagram). The nonbolded 2--spheres drawn in $N$ map homeomorphically onto the corresponding 2--sphere in $\mathbb{R}^3$.}\label{fig:3mfdstandard}
    \end{center}
    \end{figure}

       On $E_1\cong B^3$, take $f$ to be a simple 3--fold branched covering map. (The existence of such a map is standard; for example take the branch locus to consist of a trivial 2-stranded tangle, with the monodromy of the meridian about one strand permuting sheets 1,2 and the monodromy of the meridian of the other strand permuting sheets 2,3.)
       
       Now proceed for increasing $j$. When $A$ is a component of $\overline{E_j\setminus E_{j-1}}$ with two boundary components, $f$ restricts to $A$ as a 3--fold simple branched cover over the annulus $\overline{B_j\setminus B_{j-1}}$. We obtain this map $f$ by applying Lemma \ref{lem:extend-from-bdry} after choosing the restriction of $f$ to the new boundary component to be some 3--fold simple branched cover over $S^2$.

       Now suppose $P$ is a component of $\overline{E_j\setminus E_{j-1}}$ with three boundary components $\Sigma_1, \Sigma_2, \Sigma_3,$ where $\Sigma_1\subseteq E_{j-1}$ and $\Sigma_2, \Sigma_3\subseteq (E_j\setminus E_{j-1})$. Let $D_1, D_2, D_3$ be disjoint closed balls in the interior of $P$, and define $f|_{D_i}$ to be a homeomorphisms to $B_{j-1}$ for each $i=1,2,3$. On each of the two components of $\partial P\cap \partial E_j$, take $f$ to be some 3--fold simple branched cover over $S^2$. Now apply Lemma~\ref{lem:extend-from-bdry} with $A=\Sigma_1\cup\partial D_1\cup\partial D_2\cup\partial D_3$ and $B=\Sigma_2\cup \Sigma_3$ to $P\setminus\cup_{i=1}^3\mathring{D_i}$ to extend $f$ over $P$. (Note that the cover of $A$ over $S^2$ is unbranched on the components $\partial D_1, \partial D_2, \partial D_3$. This is allowed in the use of Lemma~\ref{lem:extend-from-bdry}. In fact, since $\Sigma_1\to S^2$ is a simple cover of degree 3, the cover $A\to S^2$ is still simple, of degree 6, with branch set contained in $\Sigma_1$.)

    Lemma \ref{lem:extend-from-bdry} ensures that the constructed branched covering $f:N\to\mathbb{R}^3$ is simple. 
    Let $\mathcal{P}$ be the set of all components of $\overline{E_j\setminus E_{j-1}}$ that have three boundary components  across all $j$. Then for a regular point $p\in B_1$, the preimage $f^{-1}(p)$ contains three points in $E_1$ and three points in each $P\in\mathcal{P}$. By the remark after the statement of Proposition~\ref{prop:surfacestandardex}, if $|\Ends(\Sigma)|=k$ is finite then $|\mathcal{P}|=k-1$. If $|\Ends(\Sigma)|$ is infinite then $|\mathcal{P}|=\aleph_0$. Thus, the degree of $f$ is precisely $\min\{3k,\aleph_0\}$.
\end{proof}

\begin{figure} 
\begin{center}
\labellist
\pinlabel{\small{(12)}} at 0 92
\pinlabel{\small{(23)}} at 20 92
\pinlabel{\small{(12)}} at 52 92
\pinlabel{\small{(23)}} at 74 92
\pinlabel{\small{(13)}} at 82 59
\pinlabel{\small{(12)}} at 182 92
\pinlabel{\small{(13)}} at 182 -8
\pinlabel{\small{(23)}} at 204 92
\pinlabel{\small{(12)}} at 254 92
\pinlabel{\small{(13)}} at 254 -8
\pinlabel{\small{(23)}} at 276 92
\pinlabel{\small{(12)}} at 315 92
\pinlabel{\small{(13)}} at 315 -8
\pinlabel{\small{(23)}} at 337 92
\endlabellist
\vspace{.1in}
\includegraphics[width=110mm]{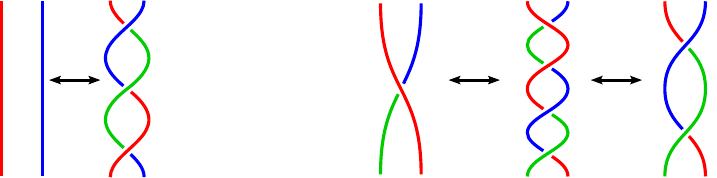}
\end{center}
\caption{Left: Montesinos move on the branch locus of a 3--fold cover. Right: Montesinos move and RII move. Applying a Montesinos move changes the branching locus but results in a diffeomorphic covering manifold~\cite{montesinos1985lectures}.}
 \label{fig:montesinos}
\end{figure}
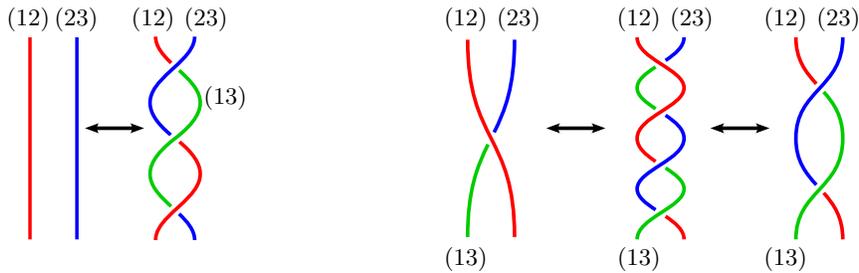

\begin{figure}
\begin{center}
\labellist
\pinlabel{\small{(12)}} at -13 138
\pinlabel{\small{(23)}} at -13 118
\pinlabel{\small{(12)}} at -13 80
\pinlabel{\small{(23)}} at -13 60
\pinlabel{\small{(12)}} at -13 20
\pinlabel{\small{(23)}} at -13 0

\pinlabel{\small{(12)}} at 262 138
\pinlabel{\small{(23)}} at 262 118
\pinlabel{\small{(12)}} at 262 80
\pinlabel{\small{(23)}} at 262 60
\pinlabel{\small{(12)}} at 262 20
\pinlabel{\small{(23)}} at 262 0
\endlabellist
\includegraphics[width=120mm]{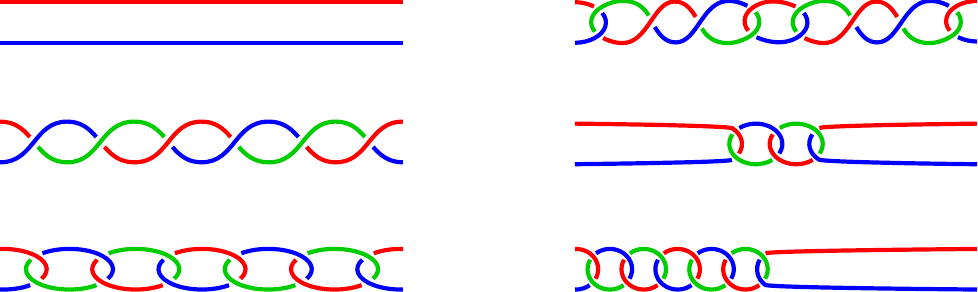}
\end{center}
\caption{An assortment of possible branch sets for a 3--fold cover $\mathbb{R}^3\to \mathbb{R}^3$, related by Montesinos moves. Any number of connected components can be realized.}
 \label{fig:R3-br-sets}
\end{figure}

\rthreeuniversal*

\begin{proof}[Proof of Corollary~\ref{rthreeuniversal}  
from Theorem \ref{thm:noncompact} and Proposition \ref{prop:r2coveringr2}]
    Let $N$ be an open 3--manifold. If $N$ has infinitely many ends, then by Theorem \ref{thm:noncompact}, $N$ admits an $\aleph_0$--fold branched cover over $\mathbb{R}^3$. If $N$ has finitely many ends, then by the same theorem there exists a finite--fold branched cover $f:N\to\mathbb{R}^3$. Let $g:\mathbb{R}^3\to\mathbb{R}^3$ be the $\aleph_0$--fold branched cover constructed in Proposition~\ref{prop:r2coveringr2}, with branch set disjoint from the branch locus of $f$. Then $g\circ f: N\to\mathbb{R}^3$ is a (nonsimple) $\aleph_0$--fold branched cover. If desired, perturb $g\circ f$ to obtain a simple branched cover.
\end{proof}

\begin{remark}\label{rem:br-set}
    The branch locus of the map constructed in the proof of Theorem \ref{thm:noncompact} is a disjoint union of embedded circles and lines. We note that noncompact components can be traded for compact ones by applying Montesinos moves on the branching set, as in the next example.
\end{remark}

\begin{example}\label{ex:R3-over-R3}
    Figure~\ref{fig:R3-br-sets} shows a family of 3--fold branched covers from  $\mathbb{R}^3$ to itself. The branch locus can be chosen to consist of two lines,  an infinite chain of unknots linked to their neighbors, or a union of finitely many linked unknots and two lines. These branching sets are related by Montesinos moves, shown in Figure \ref{fig:montesinos}.  A variation of the same construction allows us to realize $\mathbb{R}^3$ as an $n$--fold cover of itself, branched along a collection of embedded lines and (finitely or infinitely many) unknots. 
\end{example}

\section{Further questions} \label{sec:open-qs}
In this section, we suggest some questions that arose in the course of writing this paper. The authors would be interested to learn of solutions to these open questions.

\begin{question}
    Is $\mathbb{R}^m$ a universal base for all $m$?
\end{question}

\begin{question}
    Which open $m$--manifolds, including orientable and non-orientable ones, are universal bases?
\end{question}

\begin{question}
Does there exist a closed, oriented $m$--dimensional manifold $M$ with $\pi_1(M)$ finite such that $M$ a universal $k$--base for some $k<m|\pi_1(M)|$? 
\end{question}

 We show in Corollary~\ref{cor:nonorientodd} that compact non-orientable 3--manifolds are not covered by $S^3$ hence are not universal $n$--bases, in the most general sense possible, for any $n$. However, we can still ask: 

\begin{question}
Which compact non-orientable 3--manifolds are branch covered by all non-orientable 3--manifolds? Similarly, which compact non-orientable 4--manifolds are covered by all non-orientable 4--manifolds?
\end{question}

It was recently shown~\cite{bais2025branched} that  $\mathbb{RP}^4$ is covered by all non-orientable compact manifolds, even though it is not a universal $n$-base for any fixed $n$; and that $S^1\widetilde{\times}S^3$, the twisted $S^3$ bundle over $S^1$, is not covered by all non-orientable compact 4--manifolds.

As noted earlier, by a famous theorem of Alexander~\cite{alexander1920note}, every compact oriented PL $m$--manifold is a branched cover over $S^m$, but without restricting the number of sheets. It is unknown whether there exists some $n_m$ such that $S^m$ is a universal $n_{m}$--base when $m\ge 5$. 
The compact version of the following question has been kicking around for a couple of decades (see for example \cite[Problem~C]{piergallini1995four}).

\begin{question}\label{question:spheres}\leavevmode
\begin{enumerate}
    \item Is $S^m$ a universal $n_m$--base for all $m$? If so, can we take $n_m=m$?
  \item\label{question:spherespartb}  Is $\mathbb{R}^m$ a universal $\min\{n_m\cdot k,\aleph_0\}$--base among $m$--manifolds with $k$ ends for $m>4$? If so, what are the values of $n_m$?
\end{enumerate}
\end{question}
\noindent The case $n=4$ is answered affirmatively in~\cite{piergallinizuddas_open}. Note that they show an open 4--manifold $M$ is a $\min\{3|\Ends(M)|,\aleph_0\}$--sheeted cover of $\mathbb{R}^4$, rather than a $\min\{4|\Ends(M)|,\aleph_0\}$--sheeted cover, as one might guess from $S^4$ being a universal 4--base but not a universal 3--base. For this reason, it is hard to predict what values of $n_m$ should be expected in a potential answer to Question~\ref{question:spheres}\eqref{question:spherespartb}.

\begin{remark} The sphere $S^m$ is certainly not a universal $n$--base for $n<m$ since the length of the reduced cohomology ring of the $m$-torus $T^m$ is $m$; therefore, a branched cover $f: T^m\to S^m$ has degree at least $m$~\cite{berstein1978degree}. \end{remark}
\begin{remark}
        In low dimensions, we can arrange for branch loci to be embedded submanifolds \cite{hilden1976three,iori2002}. 
{Iori and Piergallini \cite{iori2002} showed that every 4--manifold is a 5--fold branched cover of $S^4$ with branch locus an embedded surface.}
However, this fails in higher dimensions, even in the topological category. For example, Berstein--Edmonds \cite{berstein1978degree} showed that many spin manifolds cannot be finite--fold branched covers of a sphere with embedded branch loci. For instance, they show explicitly that for $m\geq 1$ the quaternionic projective space $\mathbb{H}P^{2m}$ is not realizable as a finite--fold branched cover over the sphere $S^{8m}$ with branch locus a locally flat submanifold. It is natural to wonder which manifolds are universal $n$--bases if additional restrictions are made on the branching loci.
    \end{remark}

\begin{figure}
\includegraphics[width=140mm]{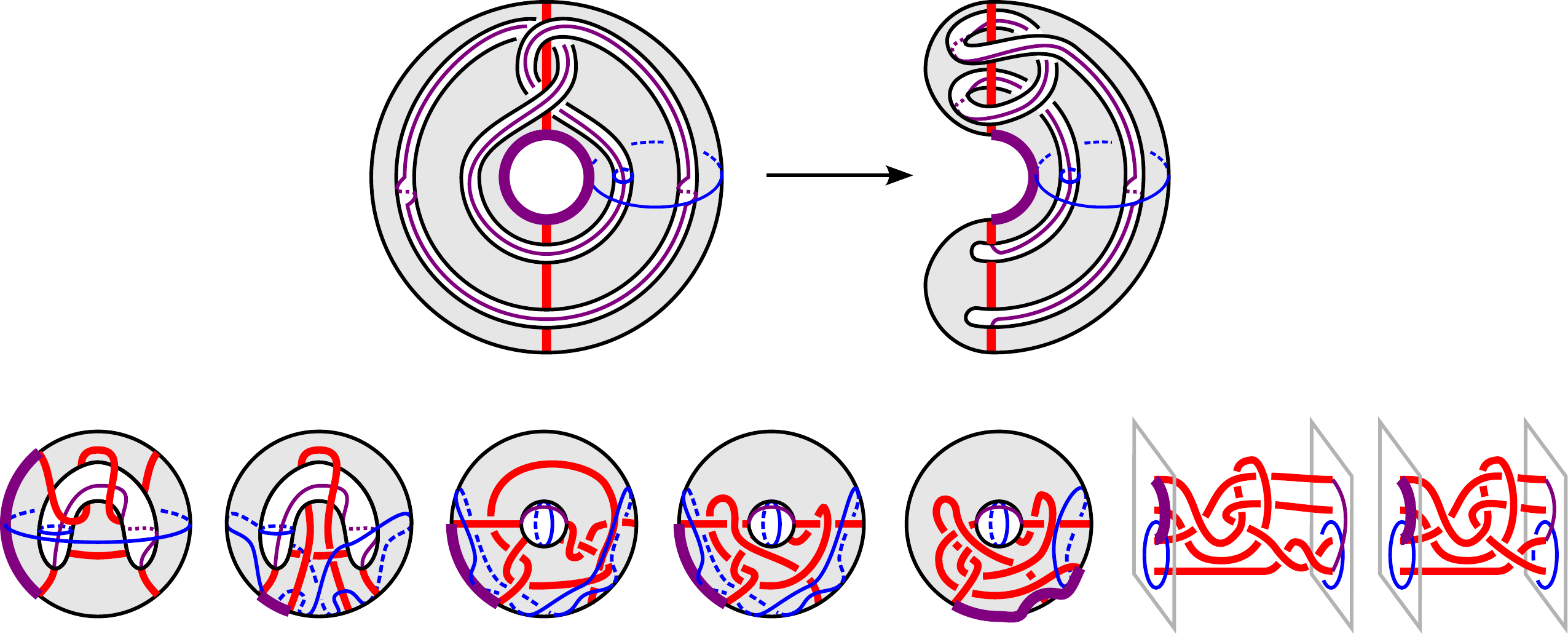}
\caption{ Top: A degree--2 branched covering of a Whitehead link complement over $S^2\times I$, where the latter is represented as a 3-ball (gray) with the neighborhood of an arc (white) deleted. The vertical red line is the branch set. The Whitehead manifold can be built as a union of a solid torus, $S^1_0\times D^2_0,$ and a countable family of Whitehead link complements $X_1,X_2,\ldots$. The bold (purple) circle indicates the curve in $\partial X_1$ that bounds a disk in $S^1_0\times D^2_0$. We also highlight a longitude on the ``inner" boundary component. The other highlighted boundary curves (blue) illustrate the gluing of $X_j, X_{j+1}$ -- the ``inner" curve in $X_j$ is glued to the ``outer" curve in $X_{j+1}$. \\
Bottom: We successively redraw the top right diagram in order to visualize the branch locus as a tangle in $S^2\times I$, where the latter is eventually represented in a standard way, with the product structure readily apparent. Note that the highlighted (blue) curves in $S^2\times \partial(I)$ are horizontal translates of each other. This is essential in concatenating these braid segments, since in the cover this means that successive circles will be identified to each other, guaranteeing that the construction of the Whitehead manifold is compatible with the branched covering maps from each $X_i$ to its corresponding copy of $S^2\times I$.}\label{fig:whitehead}
\end{figure}

\begin{example}\label{ex:whitehead}
In 1935, Whitehead~\cite{whitehead1935certain} constructed $W,$ later called {\it the Whitehead manifold,} the first example of a contractible open 3--manifold which is not homeomorphic to $\mathbb{R}^3$. The manifold $W$ is built as an ascending union of solid tori, $W=\cup_i (S^1_i\times D^2_i), i=0, 1, \dots$, where each $S^1_i\times D^2_i$ is the exterior of an unknot $U_i\subset S^3$ and $U_i$ sits inside a neighborhood of $U_{i-1}$ as the Whitehead pattern (Figure~\ref{fig:whitehead} top left). One sees from this description that $W$ may also be written as the union of the first solid torus, $S^1_0\times D^2_0,$ and a countable family of Whitehead link complements $X_j, j=1, 2, \dots$, attached to each other successively along tori $S^1\times S^1$. Precisely, for $j>1$, the ``outer'' (a distinction without a difference) boundary component of $W_j$ is attached to the ``inner'' boundary component of $W_{j-1}$ via the identification described in Figure~\ref{fig:whitehead}. (For $j=1,$ attach the outer boundary of $X_1$ to $\partial(S^1_0\times D^2_0)$, again as indicated in the figure.) The latter description is used in building our 2-- and 3--fold covers of $W$ over $\mathbb{R}^3$.) Stabilizing to a degree--3 cover allows us to arrange for the branch set to be a link of countably many compact components.

\begin{remark}
We note that there is a quiet ambiguity in the construction given in our Example~\ref{ex:whitehead}. We have regarded the Whitehead manifold as built by iteratively attaching copies of $X_{i}$ to $S^1_0\times D^2_0\bigcup(\cup_{j=1}^{i-1} X_j)$, a manifold with $S^1\times S^1$ boundary. If $X_i$ were a solid torus, in order to determine the attachment, it would suffice to specify the image of a meridian. In our case, in order to pin down {\it the} Whitehead manifold, we technically ought to specify the images of the oriented meridian and longitude of the component of $\partial(X_i)$ along which we perform the attachment. Even the orientation of the meridian plays a role: while $X_i$ admits an involution reversing the orientation of the meridian, the two choices of orientation force distinct identifications of the endpoints or the branching set, when stacking adjacent tangles in $S^2\times I$, in order for the gluing of each $X_i$ to be compatible with the covering map. Thus, if we change the orientation of the meridian on the relevant boundary component of $X_i$, this would result in swapping the ``inside'' and ``outside'' of the equatorial (blue) circle in $S^2\times I$, altering the induced identification between the tangles and thus the total branching set. Since there is not a unique choice of identifying the boundary torus so as to produce a contractible manifold, we are not too careful to remain faithful to the original construction. Specifically, if we denote $S^3\backslash \large (S^1_0\times D^2_0\bigcup(\cup_{j=1}^{i} X_j)\large) =: U_i,$ then each $U_i$ is an unknotted solid torus in $S^3$ and ``the'' Whitehead manifold can also be written as the union of the nested tori $S^1_0\times D^2_0 \subset U_1\subset U_2\subset ...$. For contractibility, it suffices that, for each $i$, the longitude of $U_i$ is nullhomotopic in $U_{i+1}$;  this is the case for the nested union pictured. Indeed, Whitehead's construction has been generalized far more broadly. For example, McMillan constructed uncountably many pairwise distinct contractible open 3--manifolds~\cite{mcmillan1962some}. 
\end{remark}

\begin{remark}
    It is shown in~\cite[Theorem~1]{mcmillan1961cartesian} that every open contractible manifold can be written as an ascending union of handlebodies. Thus, a construction analogous to ours can be applied to all contractible open manifolds.   
\end{remark}

In~\cite{montesinos2003open}, Montesinos gives a construction of uncountably many 2--fold covering maps of certain contractible open manifolds, including $W$, over $\mathbb{R}^3$. The branching loci are lines obtained by removing the wild points of knots; the knots are branch sets of covers over $S^3$ by the compactifications of the open manifolds in question. 
This construction makes use of symmetry in certain examples. We do not know if every  contractible 3--manifold $M$ admits a 2--fold branched covering over $\mathbb{R}^3$.   
\end{example}

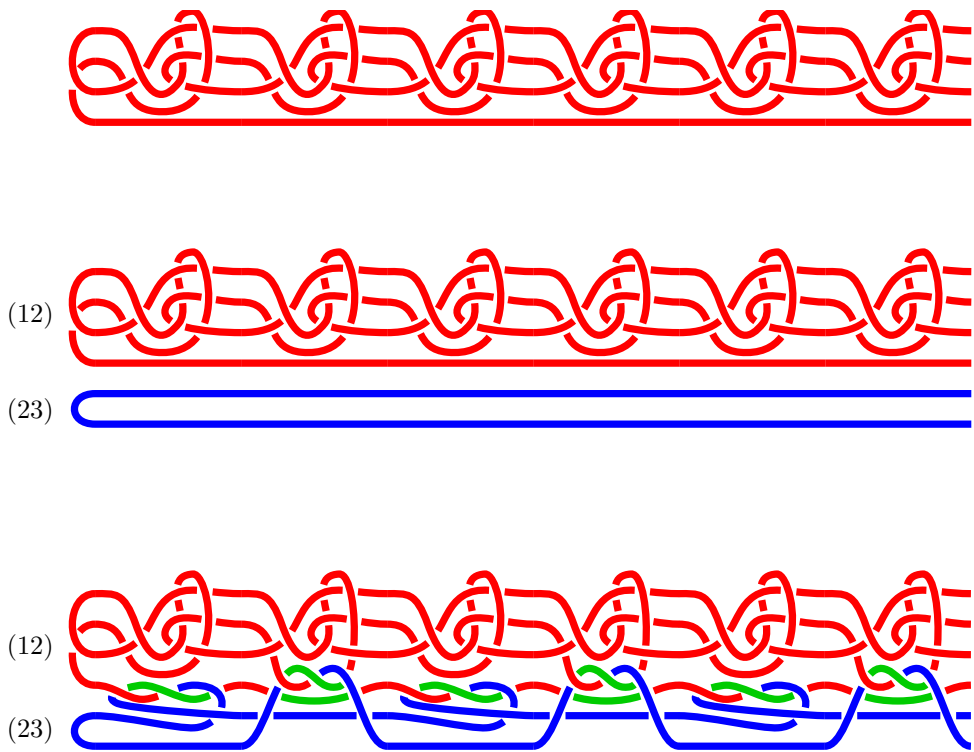
\begin{figure}
\labellist
\pinlabel{(12)} at -30 340
\pinlabel{(23)} at -30 265
\pinlabel{(12)} at -30 80
\pinlabel{(23)} at -30 15
\endlabellist
\includegraphics[width=120mm]{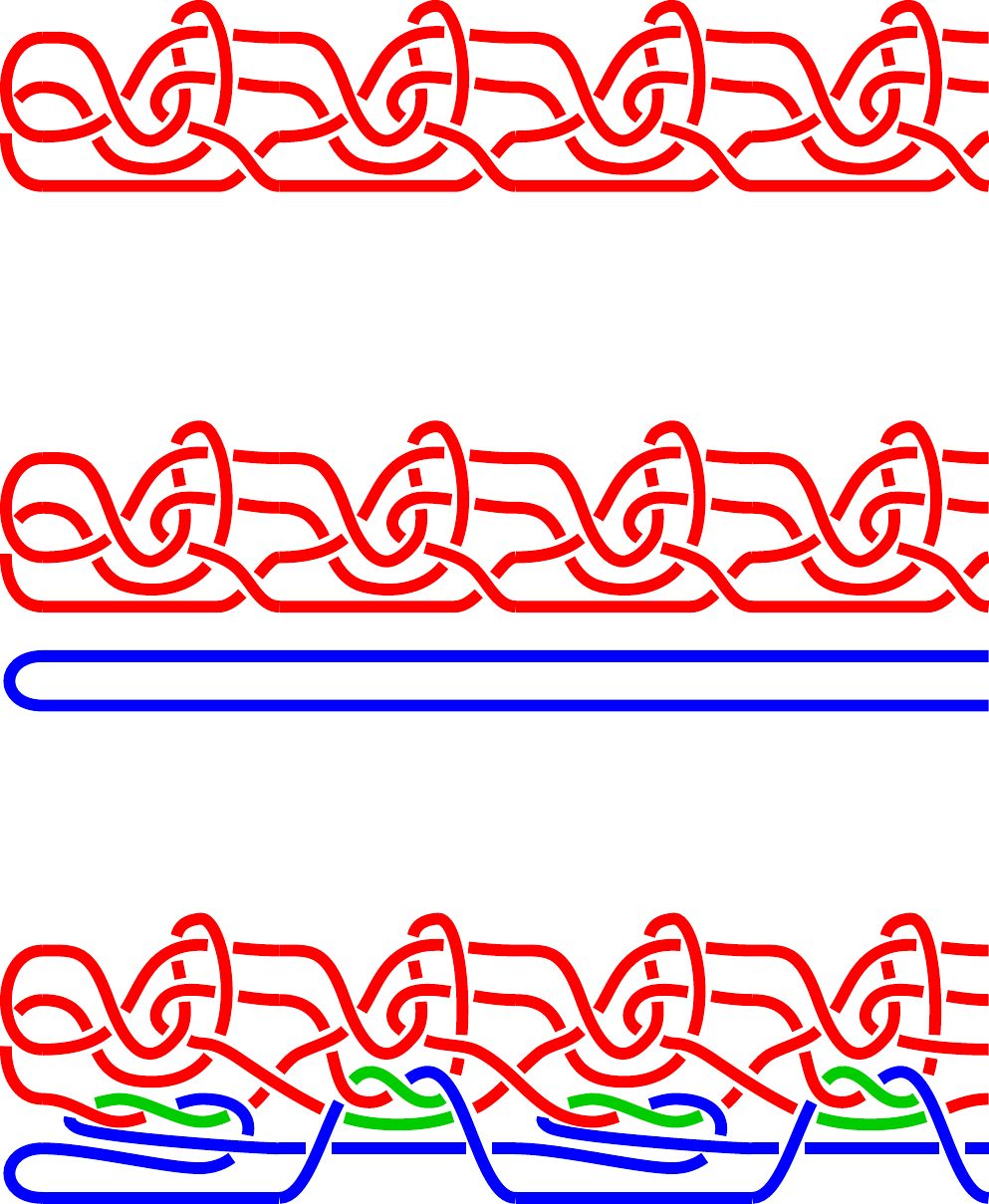}
\caption{Top: A knotted line in $\mathbb{R}^3$, the double branched cover of which is the Whitehead manifold. This line is obtained by concatenating instances of the tangle in Figure \ref{fig:whitehead} (bottom right). Center: We stabilize to a degree--3 branched covering of the Whitehead manifold over $\mathbb{R}^3$. Bottom: We apply the Montesinos move infinitely many times (preceded, alternately, by RI moves and pairs of RII moves, left to right) to transform the branch locus into a union of infinitely many circles.}\label{fig:whitehead2}
\end{figure}

\begin{question}
    Is there a contractible 3--manifold $M$ that does not admit a 2--fold covering map over~$\mathbb{R}^3$?
\end{question}

If a manifold $M^3$ as above exists, one would expect an obstruction to the existence of a 2--fold branched cover $M\to\mathbb{R}^3$ to come from the fundamental groups at ends of $M$ (Definition~\ref{def:pi-1-end}). For now, such an obstruction appears elusive.

\section*{Appendix: What we talk about when we talk about branched covers} \label{sec:appendix}

We offer a brief discussion comparing the different definitions of a ``branched cover" that one encounters in the literature. 
The concept of a branched covering has its roots in the theory of Riemann surfaces, where they arise from the study of analytic maps. In higher dimensions branched coverings are formulated naturally for PL manifolds, and were used by Heegaard~\cite{heegaard1898forstudier}, Tietze~\cite{tietze1908topologischen}, Alexander~\cite{alexander1920note,alexander1926types,alexander1928topological}, and Reidemeister~\cite{reidemeister1932knoten}.  A completely topological generalization of branched coverings was not achieved until work of Fox~\cite{fox1957covering}, who defined them as a certain class of maps between locally-connected $T_1$-spaces.

When $M$ and $N$ are compact manifolds (possibly with boundary), one may also define a branched cover as a surjective, finite-to-one, proper, open map~\cite{bersteinontheconstruction}. When the manifolds involved are noncompact, it is reasonable to waive the assumption that the degree is finite. One may adopt the following definition.

\begin{definition}
\label{def:branchedcovering2}
Let $N$ and $M$  be smooth connected $m$--dimensional manifolds.  A smooth map $f:N \rightarrow M$ is a \emph{branched covering} if it is a surjective, open map, such that for any $y \in M$ the set $f^{-1}(y)$ is a totally disconnected subset of $N$, and for any proper compact set $C \subsetneq M$, each connected component of $f^{-1}(C)$ is compact. As before, the \emph{branch set} of $f$ is the set $B_f \subseteq N$ where $f$ fails to be a local homeomorphism, and $f(B_f) \subseteq M$ is its \emph{branch locus}. 
We assume that $B_f$ is nonempty, otherwise $f$ is just an ordinary (unbranched) covering map. 
\end{definition}

Although not phrased in the language of Fox, in the smooth category this definition is equivalent to that of \cite{fox1957covering}.  We refer the interested reader to \cite{fox1957covering} or \cite{montesinos2005branched} for the relevant definitions, and briefly outline the connection between these two notions of branched coverings.

When $M$ and $N$ are manifolds, by~\cite[Corollary 4.7]{montesinos2005branched} the condition that $f^{-1}(y)$ is totally disconnected for all $y \in M$ is equivalent to $f:N\rightarrow M$ being a \emph{spread}.  Furthermore, by \cite[Theorem 5.6]{montesinos2005branched} and the fact that the components of $f^{-1}(C)$ are compact we can see that $f$ is in fact a \emph{complete spread}.  The compactness condition can also be seen to be a necessary condition for completeness in this setting.  Note also that Fox's branched coverings are necessarily open maps, whenever the base is first-countable \cite[Corollary 9.12]{montesinos2005branched}. That the branched coverings of Definition~\ref{def:branchedcovering2} satisfy the remaining requirements of Fox's definition \cite[Definition 10.1]{montesinos2005branched} follows from the discussion below.

Let $f:N\rightarrow M$ be a branched covering in the sense of Definition~\ref{def:branchedcovering2}. By \cite[Corollary 2.3]{church1963differentiable} the function $f$ will be a local homeomorphism away from a nonempty codimension--2 set $B_f\subset N$.  
Restricting $f$ away from this set gives an ordinary (unbranched) covering map $$f|_{N \setminus f^{-1}(f(B_f))} : N \setminus f^{-1}(f(B_f))\rightarrow M \setminus f(B_f).$$  

By \cite[Theorem 2.1]{church1963differentiable}  there is a closed subset $E \subseteq B_f$ with $\dim E \leq \dim N - 3$, such that every point of $N \setminus E$ has a neighborhood on which $f$ is topologically equivalent to the map 
\begin{align*}
\mathbb{C} \times \mathbb{R}^{m-2} &\rightarrow \mathbb{C} \times \mathbb{R}^{m-2}\\
(z , x_2, \ldots , x_m) &\mapsto (z^d , x_2, \ldots , x_m)
\end{align*}
for some $d \in \{1,2,3,\ldots \}$.  The integer $d$ is the local degree of the cover at the corresponding component of $f^{-1}(U)$. When $E\neq \emptyset$, in a neighborhood of a point in $E$ the map $f$ is equivalent to the cone on a branched covering $S^{m-1}\to S^{m-1}$.

When the branching locus $B_f$ of a branched cover $f: N^m\to M^m$ is a locally flat submanifold of $M$, we also have that the restriction $f_|:{f^{-1}(B_f)\to B_f}$ is an ordinary (unbranched) covering map~\cite[Lemma~3.1]{berstein1978degree}.

\subsection*{Acknowledgements}
MH is grateful to the Max Planck Institute for Mathematics in Bonn and the Dublin Institute for Advance Studies for hosting him during part of the preparation of this work.  MH also thanks Valentina Bais for helpful comments on an early draft of this paper. AK is indebted to Aru Ray for being the first to entice her into the bewildering woods of noncompact manifolds. This project began in Fall 2022 when AK and MM visited MH at Brigham Young University.

\begin{center}
  \rule{5cm}{0.4pt}
\end{center}

\vspace{.2in}
{\noindent \bf{\fontsize{23pt}{5pt}\selectfont W h i t e h e a d \hspace{.35cm} M a n i f o l d} }

\vspace{.2in}

\noindent {\it \fontsize{19.5pt}{5pt}\selectfont Layers of space in each other’s quiddity}\\
\noindent {\it  {\it \fontsize{16pt}{5pt}\selectfont \textls[20]{Extending inward beyond timidity}}\\
\noindent {\it \fontsize{15pt}{5pt}\selectfont  \textls[150]{With a splat and a wink}}\\
\noindent {\it \fontsize{14pt}{5pt}\selectfont \textls[150]{To a point it shrinks}}\\
\noindent {\it \fontsize{9pt}{3pt}\selectfont \textls[-20]{Hungry loops hunt the neck of infinity}}\\

\bibliographystyle{amsalpha}
\bibliography{main}

\providecommand{\bysame}{\leavevmode\hbox to3em{\hrulefill}\thinspace}
\providecommand{\MR}{\relax\ifhmode\unskip\space\fi MR }
\providecommand{\MRhref}[2]{%
  \href{http://www.ams.org/mathscinet-getitem?mr=#1}{#2}
}
\providecommand{\href}[2]{#2}
\begin{thebibliography}{AFWF15}

\bibitem[AB26]{alexander1926types}
James~W Alexander and Garland~B Briggs, \emph{On types of knotted curves},
  Annals of Mathematics \textbf{28} (1926), no.~1/4, 562--586.

\bibitem[AFWF15]{aschenbrenner20153}
Matthias Aschenbrenner, Stefan Friedl, Henry Wilton, and Stefan Friedl,
  \emph{3-manifold groups}, vol.~20, European Mathematical Society Z{\"u}rich,
  2015.

\bibitem[Ale20]{alexander1920note}
James Alexander, \emph{Note on {R}iemann spaces}, Bulletin of the American
  Mathematical Society \textbf{26} (1920), no.~8, 370--372.

\bibitem[Ale28]{alexander1928topological}
James~W Alexander, \emph{Topological invariants of knots and links},
  Transactions of the American Mathematical Society \textbf{30} (1928), no.~2,
  275--306.

\bibitem[BE78]{berstein1978degree}
Israel Berstein and Allan Edmonds, \emph{The degree and branch set of a
  branched covering}, Inventiones Mathematicae \textbf{45} (1978), no.~3,
  213--220.

\bibitem[BE79a]{bersteinontheconstruction}
Israel Berstein and Allan~L. Edmonds, \emph{On the construction of branched
  coverings of low-dimensional manifolds}, Trans. Amer. Math. Soc. \textbf{247}
  (1979), 87--124. \MR{517687}

\bibitem[BE79b]{berstein1979construction}
Israel Berstein and Allan~L Edmonds, \emph{On the construction of branched
  coverings of low-dimensional manifolds}, Transactions of the American
  Mathematical Society \textbf{247} (1979), 87--124.

\bibitem[BPZ25]{bais2025branched}
Valentina Bais, Riccardo Piergallini, and Daniele Zuddas, \emph{Branched
  covering representation of non-orientable $4 $-manifolds}, arXiv preprint
  arXiv:2509.09319 (2025).

\bibitem[Chu63]{church1963differentiable}
Philip~T Church, \emph{Differentiable open maps on manifolds}, Transactions of
  the American Mathematical Society \textbf{109} (1963), no.~1, 87--100.

\bibitem[Cle73]{clebsch1873theorie}
Alfred Clebsch, \emph{Zur theorie der riemann'schen fl{\"a}che}, Mathematische
  Annalen \textbf{6} (1873), no.~2, 216--230.

\bibitem[Fox57]{fox1957covering}
Ralph Fox, \emph{Covering spaces with singularities}, Algebraic Geometry and
  Topology: A Symposium in Honor of S. Lefschetz, 1957, pp.~243--257.

\bibitem[Hee98]{heegaard1898forstudier}
Poul Heegaard, \emph{Forstudier til en topologisk teori for de algebraiske
  fladers sammenhaeng}, Bojesen, 1898.

\bibitem[Hil74]{hilden1974every}
Hugh Hilden, \emph{Every closed orientable 3-manifold is a 3-fold branched
  covering space of ${S}^3$}, Bulletin of the American Mathematical Society
  \textbf{80} (1974), no.~6, 1243--1244.

\bibitem[Hil76]{hilden1976three}
Hugh~M Hilden, \emph{Three-fold branched coverings of ${S}^3$}, American
  Journal of Mathematics (1976), 989--997.

\bibitem[Hir74]{hirsch1974offene}
Ulrich Hirsch, \emph{{\"U}ber offene abbildungen auf die 3-sph{\"a}re},
  Mathematische Zeitschrift \textbf{140} (1974), no.~3, 203--230.

\bibitem[HR96]{hughes1996ends}
Bruce Hughes and Andrew Ranicki, \emph{Ends of complexes}, no. 123, Cambridge
  university press, 1996.

\bibitem[IP02]{iori2002}
Massimiliano Iori and Riccardo Piergallini, \emph{4-manifolds as covers of the
  4-sphere branched over non-singular surfaces}, Geometry and Topology
  \textbf{6} (2002), no.~1, 393--401.

\bibitem[L{\"{u}}r71]{luroth}
J.~L{\"{u}}roth, \emph{Note \"{u}ber {V}erzweigungsschnitte und {Q}uerschnitte
  in einer {R}iemann'schen {F}l\"{a}che}, Math. Ann. \textbf{4} (1871), no.~2,
  181--184. \MR{1509744}

\bibitem[MA02]{montesinos2002representing}
Jos{\'e}~Mar{\'\i}a Montesinos-Amilibia, \emph{Representing open 3-manifolds as
  3-fold branched coverings.}, Revista Matem{\'a}tica Complutense \textbf{15}
  (2002), no.~2, 533--542.

\bibitem[MA03]{montesinos2003open}
Jos{\'e}~Mar{\i}a Montesinos-Amilibia, \emph{Open 3-manifolds, wild subsets of
  s3 and branched coverings}, Revista matem{\'a}tica complutense \textbf{16}
  (2003), no.~2, 577--600.

\bibitem[MA05]{montesinos2005branched}
Jos{\'e}~Mar{\'\i}a Montesinos~Amilibia, \emph{Branched coverings after {F}ox},
  Boletin de la Sociedad Matematica Mexicana. Tercera Serie \textbf{11} (2005),
  no.~1, 19--64.

\bibitem[McM62]{mcmillan1962some}
DR~McMillan, \emph{Some contractible open 3-manifolds}, Transactions of the
  American Mathematical Society \textbf{102} (1962), no.~2, 373--382.

\bibitem[MJ61]{mcmillan1961cartesian}
DR~McMillan~Jr, \emph{Cartesian products of contractible open manifolds}.

\bibitem[Mon74]{montesinos1974representation}
Jos{\'e}~Mar{\'\i}a Montesinos, \emph{A representation of closed orientable
  3-manifolds as 3-fold branched coverings of ${S}^3$}, Bulletin of the
  American Mathematical Society \textbf{80} (1974), no.~5, 845--846.

\bibitem[Mon85]{montesinos1985lectures}
\bysame, \emph{Lectures on 3-fold simple coverings and 3-manifolds},
  Contemporary Mathematics \textbf{44} (1985), 157--177.

\bibitem[Nov]{novikov2007topological}
Serge{\u\i} Novikov, \emph{Topological library: Spectral sequences in
  topology}.

\bibitem[Per02]{perelman2002entropy}
Grisha Perelman, \emph{The entropy formula for the ricci flow and its geometric
  applications}, arXiv preprint math/0211159 (2002).

\bibitem[Per03a]{perelman2003finite}
\bysame, \emph{Finite extinction time for the solutions to the ricci flow on
  certain three-manifolds}, arXiv preprint math/0307245 (2003).

\bibitem[Per03b]{perelman2003ricci}
\bysame, \emph{Ricci flow with surgery on three-manifolds}, arXiv preprint
  math/0303109 (2003).

\bibitem[Pie89]{piergallini1989manifolds}
Ricardo Piergallini, \emph{Manifolds as branched covers of spheres.},
  Disertaciones Matem{\'a}ticas del Seminario de Matem{\'a}ticas Fundamentales
  \textbf{3} (1989), 1--25.

\bibitem[Pie92]{piergallini1992covering}
Riccardo Piergallini, \emph{Covering homotopy 3-spheres}, Comm. Math. Helv
  \textbf{67} (1992), 287--292.

\bibitem[Pie95]{piergallini1995four}
\bysame, \emph{Four-manifolds as 4-fold branched covers of ${S}^4$}, Topology
  \textbf{34} (1995), no.~3, 497--508.

\bibitem[PZ19a]{piergallini2019branched}
Riccardo Piergallini and Daniele Zuddas, \emph{On branched covering
  representation of 4-manifolds}, Journal of the London Mathematical Society
  \textbf{100} (2019), no.~1, 1--16.

\bibitem[PZ19b]{piergallinizuddas_open}
\bysame, \emph{On branched covering representation of 4-manifolds}, J. Lond.
  Math. Soc. (2) \textbf{100} (2019), no.~1, 1--16. \MR{3999680}

\bibitem[RR32]{reidemeister1932knoten}
K~Reidemeister and K~Reidemeister, \emph{Knoten und gruppen}, Knotentheorie
  (1932), 41--69.

\bibitem[Sul04]{sullivan}
Dennis Sullivan, \emph{Ren\'{e} {T}hom's work on geometric homology and
  bordism}, Bull. Amer. Math. Soc. (N.S.) \textbf{41} (2004), no.~3, 341--350.
  \MR{2058291}

\bibitem[Tho52]{thom1}
Ren\'{e} Thom, \emph{Espaces fibr\'{e}s en sph\`eres et carr\'{e}s de
  {S}teenrod}, Ann. Sci. \'{E}cole Norm. Sup. (3) \textbf{69} (1952), 109--182.
  \MR{0054960}

\bibitem[Tho54]{thom1954quelques}
Ren{\'e} Thom, \emph{Quelques propri{\'e}t{\'e}s globales des vari{\'e}t{\'e}s
  diff{\'e}rentiables}, Commentarii Mathematici Helvetici \textbf{28} (1954),
  no.~1, 17--86.

\bibitem[Tho58]{thom2}
R.~Thom, \emph{Les classes caract\'{e}ristiques de {P}ontrjagin des
  vari\'{e}t\'{e}s triangul\'{e}es}, Symposium internacional de topolog\'{\i}a
  algebraica {I}nternational symposium on algebraic topology, Universidad
  Nacional Aut\'{o}noma de M\'{e}xico and UNESCO, Mexico City, 1958,
  pp.~54--67. \MR{0102071}

\bibitem[Tie08]{tietze1908topologischen}
Heinrich Tietze, \emph{{\"U}ber die topologischen invarianten mehrdimensionaler
  mannigfaltigkeiten}, Monatshefte f{\"u}r Mathematik und Physik \textbf{19}
  (1908), 1--118.

\bibitem[Whi35]{whitehead1935certain}
John~HC Whitehead, \emph{A certain open manifold whose group is unity}, The
  Quarterly Journal of Mathematics (1935), no.~1, 268--279.

\end{thebibliography}
\end{document}